\DeclareSymbolFont{AMSb}{U}{msb}{m}{n}
\documentclass[11pt,a4paper,leqno,noamsfonts]{amsart}
   \makeatletter 
   \renewcommand\@biblabel[1]{#1.}
   \makeatother
\usepackage[english]{babel}
\usepackage[dvipsnames]{xcolor}
\usepackage{graphicx,pifont,soul}
\usepackage[utopia]{mathdesign}
\usepackage[a4paper,top=3cm,bottom=3cm,
    left=3cm,right=3cm,marginparwidth=60pt]{geometry}
\usepackage[utf8]{inputenc}
\usepackage{braket,caption,comment,mathtools,stmaryrd}
\usepackage[usestackEOL]{stackengine}
\usepackage{multirow,booktabs,microtype,relsize}   
    \numberwithin{equation}{section}
\usepackage[colorlinks,bookmarks]{hyperref} %
  \hypersetup{colorlinks,%
              citecolor=britishracinggreen,%
              filecolor=black,%
              linkcolor=cobalt,%
              urlcolor=cornellred}
\usepackage[capitalise]{cleveref}
\newcommand{\mot}{\mathsf{mot}}
\DeclareSymbolFont{usualmathcal}{OMS}{cmsy}{m}{n}
\DeclareSymbolFontAlphabet{\mathcal}{usualmathcal}
\DeclareMathAlphabet\BCal{OMS}{cmsy}{b}{n}

\definecolor{cornellred}{rgb}{0.7, 0.11, 0.11}
\definecolor{britishracinggreen}{rgb}{0.0, 0.26, 0.15}
\definecolor{cobalt}{rgb}{0.0, 0.28, 0.67}

\usepackage[colorinlistoftodos]{todonotes} % Notes in bubble

\newcommand{\boldit}[1]{\boldsymbol{#1}} % bold + italics

\newcommand{\dSt}{\mathbf{dSt}_{\BC}}
\newcommand{\dSpec}{\mathbf{Spec}\,}
\newcommand{\qcoh}{\mathsf{qcoh}}
\newcommand{\dgmod}{\mathrm{dg}\textrm{-}\mathrm{Mod}}
\newcommand{\dR}{\mathrm{dR}}

\newcommand{\BA}{{\mathbb{A}}}

\newcommand{\BC}{{\mathbb{C}}}

\newcommand{\BE}{{\mathbb{E}}}

\newcommand{\BL}{{\mathbb{L}}}

\newcommand{\BN}{{\mathbb{N}}}

\newcommand{\BP}{{\mathbb{P}}}
\newcommand{\BQ}{{\mathbb{Q}}}
\newcommand{\BR}{{\mathbb{R}}}

\newcommand{\BT}{{\mathbb{T}}}

\newcommand{\BZ}{{\mathbb{Z}}}

\newcommand{\CM}{{\mathcal M}}

\newcommand{\CO}{{\mathcal O}}

\newcommand{\CS}{{\mathcal S}}
\newcommand{\CT}{{\mathcal T}}

\newcommand{\CX}{{\mathcal X}}

\newcommand{\CZ}{{\mathcal Z}}

\newcommand{\simto}{\,\widetilde{\to}\,}
\newcommand{\into}{\hookrightarrow}
\newcommand{\onto}{\twoheadrightarrow}
\newcommand{\Hess}{\mathsf{Hess}}
\newcommand{\BRcrit}{\mathbb{R}\mathrm{crit}}
\newcommand{\derived}{\mathbf{D}}
\newcommand{\dd}{\mathrm{d}}
\newcommand{\RR}{\mathbf R}
\newcommand{\Id}{\operatorname{Id}}
\newcommand{\id}{\operatorname{id}}
\newcommand{\Spec}{\operatorname{Spec}}
\newcommand{\Supp}{\operatorname{Supp}}

\newcommand{\ch}{\mathsf{ch}}
\newcommand{\HH}{\mathrm{H}}
\newcommand{\OO}{\mathscr O}

\newcommand{\crit}{\operatorname{crit}}

\newcommand*{\defeq}{\mathrel{\vcenter{\baselineskip0.5ex \lineskiplimit0pt
                     \hbox{\scriptsize.}\hbox{\scriptsize.}}}=}

\DeclareMathOperator{\Hilb}{Hilb}
\DeclareMathOperator{\Quot}{Quot}

\DeclareMathOperator{\NCQuot}{ncQuot}
\DeclareMathOperator{\Sets}{Sets}
\DeclareMathOperator{\Sch}{Sch}

\DeclareMathOperator{\Coh}{Coh}
\DeclareMathOperator{\QCoh}{QCoh}
\DeclareMathOperator{\red}{red}

\DeclareMathOperator{\der}{der}

\DeclareMathOperator{\Sym}{Sym}

\DeclareMathOperator{\Tr}{Tr}

\DeclareMathOperator{\Rep}{Rep}
\DeclareMathOperator{\GL}{GL}

\DeclareMathOperator{\Mat}{Mat}
\DeclareMathOperator{\op}{op}
\DeclareMathOperator{\tr}{tr}

\DeclareMathOperator{\pr}{pr}
\DeclareMathOperator{\At}{At}

\newcommand{\fg}{{\mathfrak g}}
\newcommand{\fm}{{\mathfrak m}}
\newcommand{\lr}{\longrightarrow}
\newcommand{\bfa}{{\mathbf a}}
\newcommand{\unal}{\underline{\alpha}}
\newcommand{\unbet}{\underline{\beta}}
\newcommand{\ungam}{\underline{\gamma}}
\renewcommand{\hat}{\widehat}
\newcommand{\loc}{\mathrm{loc}}
\newcommand{\gl}{\mathfrak{gl}}
\newcommand{\cdga}{{\mathrm{cdga_{\BC}^{\leq 0}}}}
\newcommand{\dga}{\mathrm{dga}_\BC}

\DeclareFontFamily{OT1}{rsfs}{}
\DeclareFontShape{OT1}{rsfs}{n}{it}{<-> rsfs10}{}
\DeclareMathAlphabet{\curly}{OT1}{rsfs}{n}{it}
\newcommand\Ext{\operatorname{Ext}}
\newcommand\Hom{\operatorname{Hom}}
\newcommand\End{\operatorname{End}}
\newcommand{\lHom}{\mathscr{H}\kern-0.3em {o}\kern-0.2em{m}}
\newcommand{\RRlHom}{\mathbf{R}\kern-0.025em\mathscr{H}\kern-0.3em {o}\kern-0.2em{m}}
\newcommand{\lExt}{{\mathscr{E}\kern-0.2em xt}}
\newcommand{\RHom}{{\mathbf{R}\kern-0.07em\mathrm{Hom}}}



\usepackage[all]{xy}
\usepackage{tikz}
\usepackage{tikz-cd}
\usepackage{adjustbox}
\usepackage{rotating}
\usepackage{comment}
\newcommand*{\isoarrow}[1]{\arrow[#1,"\rotatebox{90}{\(\sim\)}"
]}
\usetikzlibrary{matrix,shapes,intersections,arrows,decorations.pathmorphing}
\tikzset{commutative diagrams/arrow style=math font}
\tikzset{commutative diagrams/.cd,
mysymbol/.style={start anchor=center,end anchor=center,draw=none}}
\newcommand\MySymb[2][\square]{%
  \arrow[mysymbol]{#2}[description]{#1}}
\tikzset{
shift up/.style={
to path={([yshift=#1]\tikztostart.east) -- ([yshift=#1]\tikztotarget.west) \tikztonodes}
}
}

\theoremstyle{definition}

\newtheorem*{lemma*}{Lemma}
\newtheorem*{theorem*}{Theorem}
\newtheorem*{example*}{Example}
\newtheorem*{fact*}{Fact}
\newtheorem*{notation*}{Notation}
\newtheorem*{definition*}{Definition}
\newtheorem*{prop*}{Proposition}
\newtheorem*{remark*}{Remark}
\newtheorem*{corollary*}{Corollary}

\newtheorem*{conventions*}{Conventions}
\newtheorem*{question*}{Question}

\newtheorem{definition}{Definition}[section]

\newtheorem{notation}[definition]{Notation}
\newtheorem{remark}[definition]{Remark}

\newtheoremstyle{thm} % <name> % (ambienti con dimostrazione)
        {3mm}% <Space above>
        {3mm}% <Space below>
        {\slshape}% <Body font> % 
        {0mm}% <Indent amount>
        {\bfseries}% <Theorem head font>
        {.}% <Punctuation after theorem head>
        {1mm}% <Space after theorem head>
        {}% <Theorem head spec (can be left empty, meaning 'normal')> 
\theoremstyle{thm}
\newtheorem{theorem}[definition]{Theorem}
\newtheorem{corollary}[definition]{Corollary}
\newtheorem{lemma}[definition]{Lemma}
\newtheorem{prop}[definition]{Proposition}
\newtheorem{thm}{Theorem}

\let\origmaketitle\maketitle
\def\maketitle{
  \begingroup
  \def\uppercasenonmath##1{} % this disables uppercasing title
  \origmaketitle
  \endgroup
}

\title[The d-critical structure on the Quot scheme of points of a Calabi--Yau  3-fold]{\Large{The d-critical structure \\ on the Quot scheme of points of a Calabi--Yau 3-fold}}
\author{Andrea T. Ricolfi \and Michail Savvas}
\date{}

\begin{document}
\maketitle

\begin{abstract}
The Artin stack $\CM_n$ of $0$-dimensional sheaves of length $n$ on $\BA^3$ carries two natural d-critical structures in the sense of Joyce. One comes from its description as a quotient stack $[\crit(f_n)/\GL_n]$, another comes from derived deformation theory of sheaves. We show that these d-critical structures agree. We use this result to prove the analogous statement for the Quot scheme of points $\Quot_{\BA^3}(\OO^{\oplus r},n) = \crit(f_{r,n})$, which is a global critical locus for every $r>0$, and also carries a derived-in-flavour d-critical structure besides the one induced by the potential $f_{r,n}$. Again, we show these two d-critical structures agree. Moreover, we prove that they locally model the d-critical structure on $\Quot_X(F,n)$, where $F$ is a locally free sheaf of rank $r$ on a projective Calabi--Yau $3$-fold $X$.

Finally, we prove that the perfect obstruction theory on $\Hilb^n\BA^3=\crit(f_{1,n})$ induced by the Atiyah class of the universal ideal agrees with the \emph{critical} obstruction theory induced by the Hessian of the potential $f_{1,n}$.
\end{abstract}

{\hypersetup{linkcolor=black}
\tableofcontents}

%%%%%%%%%%%%%%%%%%%%%%%%%%%%%%%%%%%%%%%%%%%%%%%%
%%%%%%%%%%%%%%%%%%%%%%%%%%%%%%%%%%%%%%%%%%%%%%%%
\section{Introduction}

%%%%%%%%%%%%%%%%%%%%%%%%%%%%%%%%%%%%%%%%%%%%%%%%
\subsection{Overview}
In \cite{Joyce1}, Joyce introduced \emph{d-critical structures} on schemes and Artin stacks, both in algebraic and analytic language.
Roughly speaking, a d-critical scheme is a scheme $X$ suitably covered by `d-critical charts', i.e.~schemes of the form $\crit (f) \into V$, where $f \in \OO(V)$ is a regular function on a smooth scheme $V$ and $\crit(f)$ denotes the scheme-theoretic vanishing locus of $\dd f \in \HH^0(V, \Omega_V)$. The compatibility between d-critical charts is governed by the behaviour of a section $s$ of a canonically defined sheaf of $\BC$-vector spaces $\CS^0_X$ living over $X$. Such a well-behaving section is called a 
`d-critical structure' on $X$ (see \Cref{subsec:d-crit-stacks} for more details). We use the notation $(X,s)$ to represent a d-critical scheme.

Joyce's d-critical loci play a central role in categorified Donaldson--Thomas theory: the compatibility between d-critical charts encoded in their structure is crucial to associate to a moduli space $M$ of sheaves on a Calabi--Yau $3$-fold a canonical\footnote{Strictly speaking, a choice of orientation data is also needed, see \cite{BBDJS} and \cite{BJM} for the precise statements.} perverse sheaf $\Phi_M$ (the so-called \emph{DT sheaf}), as well as a canonical virtual motive $\Phi_M^{\mot}$ (the so-called \emph{motivic DT invariant}).

Any d-critical chart $\crit(f)\into V$ has a canonical symmetric perfect obstruction theory in the sense of Behrend--Fantechi \cite{BFHilb}, induced by the Hessian of $f$. Since these symmetric obstruction theories do not necessarily glue as two-term complexes (see \cite[Example 2.17]{Joyce1}), a d-critical scheme $(X,s)$ cannot be equipped with a global symmetric perfect obstruction theory in general. However, the compatibility between d-critical charts is enough to ensure that these locally defined obstruction theories give rise to a slightly weaker structure on $X$ with similar properties, called an \emph{almost perfect obstruction theory}, as was shown in \cite{KiemSavvas}.

\smallbreak
On the other hand, Pantev--To\"{e}n--Vaqui\'e--Vezzosi \cite{PTVV} defined $k$-\emph{shifted symplectic structures} on derived schemes and derived Artin stacks (we recall their definition in \Cref{subsec:background_on_sss_and_d-crit}), for every $k \in \BZ$. It is explained in \cite[Section 3.2]{PTVV} that every $-1$-shifted symplectic structure $\omega$ on a derived scheme $\boldit{X}$ induces a symmetric perfect obstruction theory on the underlying classical scheme $X = t_0(\boldit{X}) \into \boldit{X}$. Furthermore, as we recall in \Cref{thm:truncation}, there is a \emph{truncation functor}
\begin{equation}\label{eqn:truncation}
\big\{-1\textrm{-shifted symplectic derived Artin stacks}\big\}\, \xrightarrow{\tau}\, \big\{\textrm{d-critical Artin stacks}\big\}
\end{equation}
which takes $(\BCal{X},\omega) \mapsto (\CX,s_\omega)$, where $\CX = t_0(\BCal{X})\into \BCal{X}$ in the underlying classical Artin stack of $\BCal{X}$ and $s_\omega \in \HH^0(\CS_{\CX}^0)$ is a natural d-critical structure constructed out of $\omega$. For a $-1$-shifted symplectic derived scheme $(\boldit{X},\omega)$, the induced symmetric perfect obstruction theory on $X = t_0(\boldit{X}) \into \boldit{X}$ is isomorphic to the almost perfect obstruction theory arising from $(X,s_\omega)$.

The structures described so far are reproduced in \Cref{Fig:Joyce_picture}, inspired from a larger picture in \cite{Joyce1}, where a dotted arrow means that the association only works locally. We will explain the question mark `?' appearing in the diagram in the next subsection.

\begin{figure}[ht]
    \centering
\begin{tikzpicture}[>=stealth,->,shorten >=2pt,looseness=.5,auto]
\node at (3.45,0.3) {truncation $\tau$};
\node at (5.25,1.1) {?};
\node[draw,align=center] (E) at (3.5,2) {derived moduli of sheaves\\ on a Calabi--Yau $3$-fold};
\node[draw,align=center] (A) at (0,0) {$-1$-shifted symplectic\\ derived schemes \cite{PTVV}};
\node[draw,align=center] (B) at (6.5,0) {Joyce's d-critical  \\ schemes \cite{Joyce1}};
\node[draw,align=center] (C) at (0,-4) {schemes with symmetric \\ perfect obstruction theory \cite{BFHilb}};
\node[draw,align=center] (D) at (6.5,-4) {schemes with almost perfect\\ obstruction theory \cite{KiemSavvas}};
\draw (A) to [bend left=0,looseness=1] (B) node [midway,above] {};
\draw (A) to [bend left=0,looseness=1] (C);
\draw (B) to [bend left=0,looseness=1] (D);
\draw (C) to [bend left=0,looseness=1] (D);
\draw (E) to [bend left=0,looseness=1] (A);
\draw (E) to [bend left=0,looseness=1] (B);
\draw [dashed] (B) to [bend left=0,looseness=1] (C) node [midway,above] {};
\draw [dashed] (B) to [bend left=10,looseness=1] (A) node [midway,above] {};
\end{tikzpicture}
    \caption{Landscape of the structures and their relations appearing in this paper.}
    \label{Fig:Joyce_picture}
\end{figure}
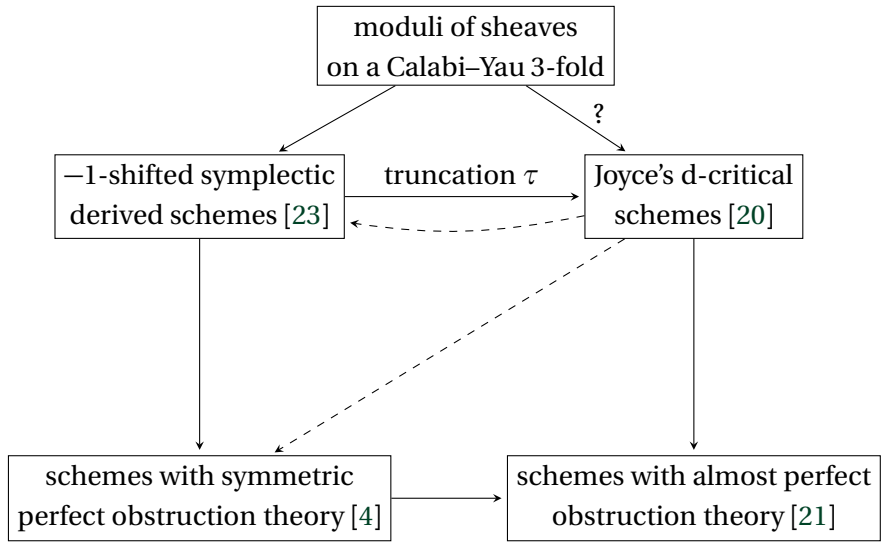

\subsection{Motivation}
A classical example of $-1$-shifted symplectic derived scheme is that of a \emph{derived critical locus} $\BRcrit(f)$, for $f \in \OO(V)$ a regular function on a smooth scheme $V$ \cite[Corollary 2.11]{PTVV}.  The space $\BRcrit(f)$ is defined as the derived fibre product of the zero section of $\Omega_V$, carrying its canonical symplectic structure, with the section $\dd f \in \HH^0(V,\Omega_V)$. The $-1$-shifted symplectic structure is denoted $\omega_{f}$ in this case. We can view the classical scheme $U=\crit(f)=t_0(\BRcrit(f))$ as a d-critical locus with d-critical structure 
\begin{equation}\label{eqn:critical_s}
s_f = f + (\dd f)^2 \in \HH^0(\CS_U^0)
\end{equation}
determined by a single d-critical chart, and the diagonal dotted arrow in \Cref{Fig:Joyce_picture} can in fact, in this special case, be filled in by means of the \emph{critical symmetric obstruction theory}
\begin{equation}\label{eqn:critical_pot}
    \BE_f = \bigl[T_V\big|_U \xrightarrow{\Hess(f)} \Omega_V\big|_U\bigr] \to \BL_U,
\end{equation}
where $\BL_Y$ is the truncated cotangent complex of a scheme $Y$. Moreover, in this case, the functor $\tau$ in \eqref{eqn:truncation} sends $(\BRcrit(f),\omega_{f})$ to $(\crit(f),s_f)$.

A further example of a $-1$-shifted symplectic derived scheme is the derived moduli scheme $\boldit{M}_X(\ch)$ of simple coherent sheaves on a projective Calabi--Yau $3$-fold $X$, with fixed Chern character $\ch \in \HH^\ast(X,\BQ)$, see \cite{PTVV}. 
In particular, the underlying classical scheme $M_X(\ch)$ is naturally a d-critical locus via the truncation functor \eqref{eqn:truncation}. However, as Behrend pointed out in \cite{Behrend_MSRI}, it is ``an embarassment of the theory'' that one cannot construct the d-critical structure on $M_X(\ch)$ directly, i.e.~without passing through derived Algebraic Geometry. This is the meaning of the question mark in \Cref{Fig:Joyce_picture}. 

This ``embarassment'' was our main motivation for starting this project. This paper dissolves such embarassment in the following sense. On the Quot scheme of $n$ points on $\BA^3$, namely the space
\begin{equation}\label{def:quot}
\mathrm{Q}_{r,n} = \Quot_{\BA^3}(\OO^{\oplus r},n) = \Set{[\OO^{\oplus r} \onto E] | \dim E = 0,\,\chi(E)=n},
\end{equation}
there are in principle \emph{two} d-critical structures:
\begin{enumerate}
    \item one arising from its critical structure $\crit(f_{r,n}) \simto \mathrm{Q}_{r,n}$ \cite{BR18} (and so a priori independent of derived geometry),\label{item1}
    \item  another arising as follows: the derived moduli stack $\BCal{M}_n$ of $0$-dimensional sheaves of length $n$ on $\BA^3$ is $-1$-shifted symplectic \cite{Brav-Dyckerhoff-II}, so by truncation one has a d-critical structure $s_n^{\der}$ on the underlying classical Artin stack $\CM_n$. Its pullback along the (smooth) morphism $\mathrm{Q}_{r,n} \to \CM_n$ forgetting the surjection defines yet another d-critical structure on $\mathrm{Q}_{r,n}$.\label{item2}
\end{enumerate}
We prove that the d-critical structures described in \eqref{item1} and \eqref{item2} agree, which shows that the d-critical structure coming from derived geometry is `morally underived', and is the simplest possible.

We also make some progress in the projective case. More precisely, let $F$ be a locally free sheaf of rank $r>0$ on a projective Calabi--Yau $3$-fold $X$. There is a natural `derived' d-critical structure on the Quot scheme of points $\Quot_X(F,n)$. We devote \Cref{sec:compact_section} to showing that such d-critical structure looks \'etale locally like the d-critical structure in \eqref{item2} above, which in turn agrees with the one induced by the unique d-critical chart on the local model $\Quot_{\BA^3}(\OO^{\oplus r},n)$.

\subsection{Main results}

We discuss in greater details our main results in the rest of this introduction.

\subsubsection{The local case}
The Quot scheme of points \eqref{def:quot}
parametrising isomorphism classes of quotients $\OO^{\oplus r} \twoheadrightarrow E$, where $E$ is a 0-dimensional sheaf of length $n$ on $\BA^3$, is proven in \cite[Theorem 2.6]{BR18} to be a global critical locus (cf.~\Cref{critical_quot}). More precisely, there is a diagram
\begin{equation}\label{diag:ncquot}
\begin{tikzcd}
\NCQuot^n_r\,\supset\,\crit(f_{r,n}) \arrow[r,"\iota_{r,n}","\sim"'] & \mathrm{Q}_{r,n}
\end{tikzcd}
\end{equation}
where $f_{r,n}$ is a regular function on the \emph{noncommutative Quot scheme} $\NCQuot^n_r$, a smooth $(2n^2+rn)$-dimensional variety which can be viewed as the moduli space of (isomorphism classes of) $(1,n)$-dimensional stable $r$-framed representations $(A,B,C,v_1,\ldots,v_r) \in \End_{\BC}(\BC^n)^3 \times (\BC^n)^r$ of the $3$-loop quiver (cf.~\Cref{fig:3loop_framed}). The function $f_{r,n}\in \OO(\NCQuot^n_r)$ is defined to be the trace of the potential $A[B,C]$. It defines a d-critical structure as in \eqref{eqn:critical_s}, namely
\[
s_{r,n}^{\crit} = s_{f_{r,n}} = f_{r,n} + (\dd f_{r,n})^2\,\in\,\HH^0\left(\CS^0_{\crit(f_{r,n})}\right).
\]
On the other hand, derived symplectic geometry endows the derived moduli stack $\BCal{M}_n$ of $0$-dimensional sheaves of length $n$ on $\BA^3$ with a $-1$-shifted symplectic structure $\omega_n$ (see \cite{Brav-Dyckerhoff-II}),
which can be truncated to produce a d-critical structure $s_n^{\der} = s_{\omega_n}$ on the (underived) moduli stack $\CM_n = t_0(\BCal{M}_n)$. The morphism $q_{r,n}\colon \mathrm{Q}_{r,n} \to \CM_n$ sending $[\OO^{\oplus r}\onto E]\mapsto [E]$ is smooth (cf.~\Cref{subsec:smoothness_forgetful}), so the pullback 
\[
s_{r,n}^{\der} = q_{r,n}^\ast s_n^{\der}
\]
defines a d-critical structure on $\mathrm{Q}_{r,n}$ (cf.~\Cref{def:derived_dcrit_quot}).

The following is our first main result.

\begin{thm}\label{main_thm}
Fix $r\geq 1$ and $n\geq 0$. There is an identity
\[
\iota_{r,n}^\ast s_{r,n}^{\der} = s_{r,n}^{\crit} \in \HH^0\left(\CS^0_{\crit(f_{r,n})}\right).
\]
\end{thm}

\subsubsection{The global case}

Consider now the case where $\BA^3$ is replaced by a smooth, \emph{projective} Calabi--Yau $3$-fold $X$ and $\OO^{\oplus r}$ by a locally free sheaf $F$ of rank $r$ on $X$. Let $\mathrm Q_{F,n} = \mathrm{Quot}_X(F,n)$ be the Quot scheme parametrising quotients $[F \twoheadrightarrow E]$ with $E$ a $0$-dimensional sheaf of length $n$ on $X$. The derived moduli stack $\BCal{M}_X(n)$ of $0$-dimensional sheaves of length $n$ on $X$ carries a canonical $-1$-shifted symplectic structure by \cite{PTVV}, which can be truncated to give a d-critical structure $s_{X,n} \in \HH^0(\CS^0_{\CM_X(n)})$. Its pullback along the forgetful morphism $\mathrm{Q}_{F,n} \to \CM_X(n)$ is a d-critical structure on the Quot scheme, denoted $s_{F,n}$, see Equation \eqref{eqn:Pullback_to_quot}.

The methods used in the proof of \Cref{main_thm} also yield (after a bit of work) the following result, which says that the d-critical scheme $(\mathrm Q_{F,n}, s_{F,n})$ is locally modelled on the d-critical scheme $(\mathrm Q_{r,n}, s_{r,n}^{\der})$.

\begin{thm}
\label{thm:main_thm_B}
Fix $r \geq 1$ and $n \geq 0$. Let $X$ be a projective Calabi--Yau $3$-fold, $F$ a locally free sheaf of rank $r$ on $X$. There exists an analytic open cover $\{ \rho_\lambda\colon T_\lambda \into \mathrm Q_{F,n} \}_{\lambda \in \Lambda}$, such that for any index $\lambda$ we have a diagram
\[
\begin{tikzcd}
& T_\lambda\arrow[swap]{dl}{\pi_\lambda}\arrow[hook]{dr}{\rho_\lambda} & \\
\mathrm{Q}_{r,n} & & \mathrm{Q}_{F,n}
\end{tikzcd}
\]
with $\pi_\lambda$ \'{e}tale, satisfying $ \pi_\lambda^* s_{r,n}^{\der} = \rho_\lambda^* s_{F,n} \in \HH^0(\CS_{T_\lambda}^0)$.
\end{thm}

\subsubsection{The two obstruction theories on \texorpdfstring{$\Hilb^n\BA^3$}{}}
Our third main result is a comparison between perfect obstruction theories on the Hilbert scheme of points $\Hilb^n\BA^3$. If $V_n=\NCQuot^n_1$ is the noncommutative Hilbert scheme, Diagram \eqref{diag:ncquot} becomes
\begin{equation}\label{eqn:Hilb_is_critical}
\begin{tikzcd}
V_n \,\supset\, \crit(f_{1,n})\arrow[r,"\iota_{1,n}","\sim"'] & 
\Hilb^n\BA^3,
\end{tikzcd}
\end{equation}
and the Hessian construction \eqref{eqn:critical_pot} defines a symmetric obstruction theory
\begin{equation}\label{eqn:crit_pot_1231}
\begin{tikzcd} \BE_{f_{1,n}} \arrow{r}{\varphi_{\crit}} & \BL_{\crit(f_{1,n})}.
\end{tikzcd}
\end{equation}
On the other hand, viewing the Hilbert scheme as a parameter space for ideal sheaves of colength $n$, one obtains the symmetric obstruction theory (more details are found in  \cite{Quot19,Equivariant_Atiyah_Class})
\[
\begin{tikzcd}\BE_{\der} = \RR \pi_\ast \RRlHom(\mathfrak I,\mathfrak I)_0[2] \arrow{r}{\varphi_{\der}} & \BL_{\Hilb^n\BA^3},
\end{tikzcd}
\]
where $\mathfrak I \subset \OO_{\BA^3\times \Hilb^n\BA^3}$ is the universal ideal sheaf, $\pi\colon \BA^3 \times  \Hilb^n\BA^3 \to \Hilb^n\BA^3$ is the second projection, and $\RRlHom(\mathfrak I,\mathfrak I)_0$ denotes the $[-1]$-shifted cone of the trace map $\RRlHom(\mathfrak I,\mathfrak I) \to \OO$. 

The following result, established in \Cref{sec:derived_hilb}, proves Conjecture 9.9 in \cite{FMR}.
\begin{thm}
\label{thm:symmetric_pot}
The isomorphism $\iota_{1,n}$ in \eqref{eqn:Hilb_is_critical} induces an isomorphism of perfect obstruction theories 
\[
\begin{tikzcd}
\iota_{1,n}^\ast \BE_{\der}\arrow{rr}{\sim}\arrow[swap]{dr}{\iota_{1,n}^\ast\varphi_{\der}} & & \BE_{f_{1,n}}\arrow{dl}{\varphi_{\crit}} \\
& \BL_{\crit(f_{1,n})} & 
\end{tikzcd}
\]
\end{thm}

%%% We also prove a global analogue of \Cref{thm:symmetric_pot}: in \Cref{thm:global_APOT} we show that if $X$ is a projective Calabi--Yau $3$-fold, then the symmetric obstruction theory on $\Hilb^nX$ induced by derived geometry is, viewed as an almost perfect obstruction theory in the sense of \cite[Definition 3.1]{KiemSavvas}, locally modelled on the critical obstruction theory $\BE_{f_{1,n}}$ given in \eqref{eqn:crit_pot_1231}.

\noindent\textbf{Conventions.}
\emph{We work over $\BC$ throughout. The `font' used in this paper for schemes, derived schemes, stacks and derived stacks will be $X$, $\boldit{X}$, $\CX$ and $\BCal{X}$ respectively. For a quasiprojective variety $X$, we let $\CM_X(n)$ denote the moduli stack of $0$-dimensional coherent sheaves of length $n$ on $X$. If $F$ is a coherent sheaf on $X$, we also set $\mathrm{Q}_{F,n} = \Quot_X(F,n)$, where the right hand side is Grothendieck's Quot scheme, parametrising quotients $F \onto E$ where $E$ is a $0$-dimensional sheaf of length $n$. When $(X,F)=(\BA^3,\OO^{\oplus r})$, we set $\CM_n=\CM_{\BA^3}(n)$ and $\mathrm{Q}_{r,n} = \Quot_{\BA^3}(\OO^{\oplus r},n)$.}

%%%%%%%%%%%%%%%%%%%%%%%%%%%%%%%%%%%%%%%%%%%%%%%%
\subsection*{Acknowledgments}
We wish to thank Pierrick Bousseau, Ben Davison, Dominic Joyce, Dragos Oprea, Sarah Scherotzke and Nicolò Sibilla and Okke van Garderen for helpful discussions. We also thank the anonymous referee for their observations and suggestions.

%%%%%%%%%%%%%%%%%%%%%%%%%%%%%%%%%%%%%%%%%%%%%%%%
%%%%%%%%%%%%%%%%%%%%%%%%%%%%%%%%%%%%%%%%%%%%%%%%
\section{Background material}

%%%%%%%%%%%%%%%%%%%%%%%%%%%%%%%%%%%%%%%%%%%%%%%%
\subsection{Shifted symplectic structures} \label{subsec:background_on_sss_and_d-crit}

Let $\cdga$ be the category of non-positively  graded commutative differential graded $\BC$-algebras. There is a spectrum functor $\dSpec \colon \cdga \to \dSt$ to the category of derived stacks (see \cite[Definition 2.2.2.14]{HAG-DAG} or \cite[Definition 4.2]{Toen_higher-derived}). An object of the form $\dSpec A$ is called an \emph{affine derived $\BC$-scheme}. Such objects provide the Zariski local charts for general \emph{derived} $\BC$-\emph{schemes}, see \cite[Section 4.2]{Toen_higher-derived}. An object $\BCal{X}$ in $\dSt$ is called a \emph{derived Artin stack} if it is $m$-geometric (cf.~\cite[Definition 1.3.3.1]{Toen_higher-derived}) for some $m$ and its `classical truncation' $t_0(\BCal{X})$ is an Artin stack (and not a higher stack). A derived Artin stack $\BCal{X}$ admits an \emph{atlas}, i.e.~a smooth surjective morphism $\boldit{U} \to \BCal{X}$ from a derived scheme. Every derived Artin stack $\BCal{X}$ has a \emph{cotangent complex} $\BL_{\BCal{X}}$ of finite cohomological amplitude in $[-m,1]$ and a dual tangent complex $\BT_{\BCal{X}}$. Both are objects in a suitable stable $\infty$-category $L_{\qcoh}(\BCal{X})$ (see \cite{Toen_higher-derived} or \cite{HAG-DAG} for its definition).

Shifted symplectic structures on derived Artin stacks were introduced by Pantev--To\"{e}n--Vaqui\'e--Vezzosi in \cite{PTVV}. The definition is given in the affine case first, and then generalised by showing the local notion satisfies smooth descent. We recall the local definition: let us set $\boldit{X} = \dSpec A$, so that $L_{\qcoh}(\boldit{X}) \cong \derived(\dgmod_A)$. For all $p\geq 0$ one can define the exterior power complex $(\Lambda^p \BL_{\boldit{X}},\dd) \in L_{\qcoh}(\boldit{X})$, where the differential $\dd$ is induced by the differential of the algebra $A$. For a fixed $k \in \BZ$, define a $k$-shifted $p$-form on $\boldit{X}$ to be an element $\omega^0 \in (\Lambda^p \BL_{\boldit{X}})^k$ such that $\dd \omega^0 = 0$. To define the notion of closedness, consider the de Rham differential $\dd_{\dR}\colon \Lambda^p \BL_{\boldit{X}} \to \Lambda^{p+1} \BL_{\boldit{X}}$. A $k$-\emph{shifted closed $p$-form} is a sequence $(\omega^0,\omega^1,\ldots)$, with $\omega^i \in (\Lambda^{p+i}\BL_{\boldit{X}})^{k-i}$, such that $\dd \omega^0 = 0$ and $\dd_{\dR} \omega^i + \dd \omega^{i+1} = 0$. When $p=2$, any $k$-shifted $2$-form $\omega^0 \in (\Lambda^2 \BL_{\boldit{X}})^k$ induces a morphism $\omega^0\colon \BT_{\boldit{X}} \to \BL_{\boldit{X}}[k]$ in $L_{\qcoh}(\boldit{X})$, and we say that $\omega^0$ is \emph{non-degenerate} if this morphism is an isomorphism in $L_{\qcoh}(\boldit{X})$.

\begin{definition}[{\cite[Definition 1.18]{PTVV}}]
A $k$-shifted closed $2$-form $\omega = (\omega^0,\omega^1,\ldots)$ is called a $k$-\emph{shifted symplectic structure} if $\omega^0$ is non-degenerate. We say that $(\boldit{X},\omega)$ is a $k$-shifted symplectic (affine) derived scheme.
\end{definition}

%%%%%%%%%%%%%%%%%%%%%%%%%%%%%%%%%%%%%%%%%%%%%%%%
\subsection{d-critical schemes and Artin stacks}\label{subsec:d-crit-stacks}

Let $X$ be a scheme over $\BC$. Joyce \cite{Joyce1} proved the existence of a canonical sheaf of $\BC$-vector spaces $\CS_X$ such that for every triple $(R,V,i)$, where $R \subset X$ is an open subscheme, $V$ is a smooth scheme and $i\colon R \into V$ is a closed immersion with ideal $\mathscr I$, one has an exact sequence
\[
\begin{tikzcd}
  0 \arrow{r} 
& \CS_X\big|_R \arrow{r}
& \OO_V/\mathscr I^2 \arrow{r}{\dd}
& \Omega_V/\mathscr I\cdot \Omega_V,
\end{tikzcd}
\]
where the last map is induced by the exterior derivative; see \cite[Theorem 2.1]{Joyce1} for the full list of properties characterising $\CS_X$. Joyce also proved the existence of a subsheaf $\CS_X^0\subset \CS_X$ and a direct sum  decomposition
\[
\CS_X = \CS_X^0\oplus \BC_X,
\]
where $\BC_X$ is the constant sheaf on $X$ and, for any triple $(R,V,i)$ as above, one has
\[
\begin{tikzcd}
\CS^0_X\big|_R = \ker\, \bigl(
\CS_X\big|_R \arrow[hook]{r}
& \OO_V/\mathscr I^2 \arrow[two heads]{r} & \OO_{R_{\red}} \bigr) \,\subset \,\CS_X\big|_R.
\end{tikzcd}
\]

If $V$ is a smooth scheme carrying a regular function $f \in \OO(V)$ with critical locus $R=\crit(f) \subset V$, and $f|_{R_{\red}} = 0$, then $\mathscr I = (\dd f) \subset \OO_V$ and a natural element of $\HH^0(\CS_X^0|_R)$ is the section $f + (\dd f)^2$.

\begin{definition}[{\cite[Definition 2.5]{Joyce1}}]\label{def:d-critical_schemes}
A \emph{d-critical scheme} is a pair $(X,s)$, where $X$ is an ordinary scheme and $s$ is a section of $\CS_X^0$ with the following property: for every point $p \in X$ there is a quadruple $(R,V,f,i)$, called a \emph{d-critical chart}, where $R \into X$ is an open neighbourhood of $p$, $V$ is a smooth scheme, $f \in \OO(V)$ is a regular function such that $f|_{R_{\red}} = 0$, having critical locus $i\colon R \into V$, and $s|_R = f + (\dd f)^2 \in \HH^0(\CS_X^0|_R)$. The section $s$ is called a \emph{d-critical structure} on $X$.
\end{definition}

By \cite[Proposition 2.8]{Joyce1}, if $g\colon X \to Y$ is a smooth morphism of schemes and $t \in \HH^0(\CS^0_Y)$ is a d-critical structure on $Y$, then $g^\ast t \in \HH^0(\CS^0_X)$ is a d-critical structure on $X$.

For an Artin stack $\CX$, Joyce defined the sheaf $\CS_{\CX}^0$ by smooth descent \cite[Corollary 2.52]{Joyce1}. Recall that to give a sheaf $\mathscr F$ on $\CX$ one has to specify an \'etale sheaf $\mathscr F(U,u)$ for every smooth $1$-morphism $u\colon U \to \CX$ from a scheme, along with a series of natural compatibilities. Similarly, to give a section $s \in \HH^0(\mathscr F)$ is to give a collection of compatible sections $s(U,u)\in \HH^0(\mathscr F(U,u))$ for every smooth $1$-morphism $u\colon U \to \CX$ from a scheme.

Joyce defined a \emph{d-critical Artin stack} (see \cite[Definition 2.53]{Joyce1}) to be a pair $(\CX,s)$, where $\CX$ is a classical Artin stack, $s \in \HH^0(\CS_{\CX}^0)$ is a section such that $s(U,u)$ defines a d-critical structure on $U$ (being a section of $\CS^0_{\CX}(U,u) = \CS_U^0$) according to \Cref{def:d-critical_schemes}, for every smooth $1$-morphism $u\colon U \to \CX$.

If $G$ is an algebraic group acting on a scheme $Y$, the sheaf $\CS^0_Y$ is naturally $G$-equivariant, so there is a well-defined subspace $\HH^0(\CS^0_Y)^G \subset \HH^0(\CS^0_Y)$ of $G$-invariant sections. If $\CX = [Y/G]$, then $\HH^0(\CS^0_{\CX}) = \HH^0(\CS^0_Y)^G$, and the d-critical structures on $\CX$ are canonically identified with the $G$-invariant d-critical structures on $Y$, see \cite[Example 2.55]{Joyce1}. We will make this identification throughout without further mention.

\begin{theorem}[{\cite[Theorem 3.18]{BBBJ}}]\label{thm:truncation}
Let $(\BCal{X},\omega)$ be a $-1$-shifted symplectic derived Artin stack. Then the underlying classical Artin stack $\CX = t_0(\BCal{X})$ extends in a canonical way to a d-critical Artin stack $(\CX,s_{\omega})$. This defines a `truncation functor' $\tau$, as in \eqref{eqn:truncation}, from the $\infty$-category of $-1$-shifted symplectic derived Artin stacks to the $2$-category of d-critical Artin stacks.
\end{theorem}

See also \cite[Theorem 6.6]{BBJ} for the analogous result proved for schemes.

%%%%%%%%%%%%%%%%%%%%%%%%%%%%%%%%%%%%%%%%%%%%%%%%
\subsection{Symmetric obstruction theories}
Let $M$ be a $\BC$-scheme with full cotangent complex $L_M^\bullet \in \derived^{(-\infty,0]}(\QCoh_M)$, and let $\BL_M \in \derived^{[-1,0]}(\QCoh_M)$ denote its cutoff at $-1$. A \emph{perfect obstruction theory} on $M$, as defined by Behrend--Fantechi \cite{BFinc}, is a pair $(\BE,\phi)$, where $\BE$ is a perfect complex of perfect amplitude contained in $[-1,0]$, and $\phi\colon \BE \to \BL_M$ a morphism in the derived category, such that $h^0(\phi)$ is an isomorphism and $h^{-1}(\phi)$ is onto. A perfect obstruction theory $(\BE,\phi)$ is called \emph{symmetric} if there exists an isomorphism $\vartheta\colon \BE \simto \BE^\vee[1]$ such that $\vartheta^\vee[1] = \vartheta$. See \cite{BFHilb} for background on symmetric obstruction theories.

If $(\boldit{M},\omega)$ is a $-1$-shifted symplectic derived scheme, with underlying classical scheme $i\colon M \into \boldit{M}$. Then $\BE = i^\ast \BL_{\boldit{M}} \to \BL_{M}$ is a perfect obstruction theory, which is furthermore symmetric thanks to the non-degenerate pairing $\vartheta = i^\ast \omega^0$. This association explains the left vertical arrow in \Cref{Fig:Joyce_picture}.

%%%%%%%%%%%%%%%%%%%%%%%%%%%%%%%%%%%%%%%%%%%%%%%%
\subsection{Smoothness of the forgetful map}\label{subsec:smoothness_forgetful}

Let $X$ be a quasiprojective variety, $n\geq 0$ an integer and $F \in \Coh(X)$ a coherent sheaf on $X$. Set $\mathrm{Q}_{F,n} = \Quot_X(F,n)$ and let $\pi\colon \mathrm{Q}_{F,n} \times X \to X$ be the projection. Forgetting the surjection defining the universal quotient
\[
[\pi^\ast F \onto \mathscr Q]\,\,\, \mapsto\,\,\, \mathscr Q
\]
defines a map $q_{F,n}\colon \mathrm{Q}_{F,n} \to \CM_X(n)$ to the moduli stack of length $n$ coherent sheaves on $X$. Such a map exists since $\mathscr Q \in \Coh(\mathrm{Q}_{F,n} \times X)$ is flat over $\mathrm{Q}_{F,n}$ by definition of the Quot scheme. We call $q_{F,n}$ the \emph{forgetful morphism} throughout.

We will need the following result.
\begin{prop}\label{prop:smoothness_forgetful}
Let $X$ be a reduced projective variety, $F$ a locally free sheaf of rank $r\geq 1$ on $X$. For every $n \geq 0$, the forgetful morphism
\[
q_{F,n}\colon \mathrm{Q}_{F,n} \to \CM_X(n)
\]
is smooth of relative dimension $rn$.
\end{prop}

Before proving the result, we recall a classical result by Grothendieck. According to \cite[Th\'eor\`eme~$7.7.6$]{EGA32} (but see also \cite[Theorem 5.7]{Nit}), if $f\colon Y\to B$ is a proper morphism to a locally noetherian scheme $B$, and $E$ is a coherent $B$-flat sheaf on $Y$, there exists a coherent sheaf $\mathscr Q_E$ on $B$ along with functorial isomorphisms
\[
\begin{tikzcd}
\eta\colon f_\ast (E\otimes_{\OO_B}\mathscr M)\arrow{r}{\sim} & \lHom_{\OO_B}(\mathscr Q_E,\mathscr M)
\end{tikzcd}
\]
for all quasicoherent sheaves $\mathscr M$ on $B$.
The sheaf $\mathscr Q_E$ is unique up to a unique isomorphism, it behaves well with respect to pullback, and moreover it is locally free exactly when $f$ is cohomologically flat in dimension $0$ \cite[Proposition 7.8.4]{EGA32}. For instance, any proper flat morphism with geometrically reduced fibres is cohomologically flat in dimension $0$ \cite[Proposition 7.8.6]{EGA32}.

\begin{proof}[Proof of Proposition~\ref{prop:smoothness_forgetful}]
Let $\CX_{F,n}$ be the stack of pairs $(E,\alpha)$, where $E$ is a $0$-dimensional sheaf of length $n$ and $\alpha$ is an $\OO_X$-linear homomorphism $\alpha\colon F \to E$. Then we have an open immersion $\mathrm{Q}_{F,n} \into \CX_{F,n}$ and the morphism $q_{F,n}$ extends to a morphism
\[
\pi_{F,n}\colon \CX_{F,n} \to \CM_X(n).
\]
Let $B$ be a scheme. By standard methods, we may reduce to the case where $B$ is locally noetherian. Given a map $B \to \CM_X(n)$, let us consider the fibre products
\[
\begin{tikzcd}
P\MySymb{dr} \arrow{r}\arrow{d} & \mathrm{Q}_{F,n}\arrow{d}{q_{F,n}}\\
B \arrow{r} &  \CM_X(n)
\end{tikzcd}
\qquad \text{and} \qquad
\begin{tikzcd}
V\MySymb{dr} \arrow{r}\arrow{d} &  \CX_{F,n}\arrow{d}{\pi_{F,n}}\\
B \arrow{r} &  \CM_X(n)
\end{tikzcd}
\]
where $B \to \CM_X(n)$ corresponds to a $B$-flat family of $0$-dimensional sheaves $E \in \Coh(X \times B)$. The associated sheaf $\mathscr Q_E \in \Coh(B)$ is locally free of rank $rn$, since the projection $f\colon X\times B \to B$ is cohomologically flat in dimension $0$ since it is proper with reduced fibres. Moreover, the projection $V \to B$ is canonically isomorphic to the structure morphism
\[
\Spec \Sym_{\OO_B}\mathscr Q_E \to B,
\]
because of Grothendieck's theorem \cite[Cor.~7.7.8, Rem.~7.7.9]{EGA32}, which says the following: Let $f\colon Y\to B$ be a projective morphism, $\mathscr F$ and $\mathscr E$ two coherent sheaves on $Y$. Consider the functor $\Sch_B^{\op}\to \Sets$ sending a $B$-scheme $T\to B$ to the set of morphism $\Hom_{Y_T}(\mathscr F_T,\mathscr E_T)$, where $\mathscr F_T$ and $\mathscr E_T$ are the pullbacks of $\mathscr F$ and $\mathscr E$ along the projection $Y_T=Y\times_BT\to Y$. Then, if $\mathscr E$ is flat over $B$, the above functor is represented by a linear scheme $\Spec \Sym_{\OO_B} \mathscr H\to B$, where $\mathscr H$ is a coherent sheaf on $B$. However, in our case we have $Y = X\times B$, $f=p_2$ the second projection and $\mathscr F=p_1^\ast F$, and the functor described above is precisely the functor of points of $V$, thus we must have $\mathscr H = \mathscr Q_E$. More details can be found in the proof of Grothendieck's result found in \cite[Theorem 5.8]{Nit}.

Therefore, since $P \into V$ is open and $V \to B$ is an affine bundle of rank $rn$, the projection $P\to B$ is smooth of relative dimension $rn$.
\end{proof}

\begin{remark}
Let $X^\circ \into X$ be an open subscheme, where $X$ is a reduced projective variety as in \cref{prop:smoothness_forgetful}, and let $F^\circ = F|_{X^\circ}$ be the restriction of a locally free sheaf $F$ over $X$. Then 
\[
q_{F^\circ,n}\colon \Quot_{X^\circ}(F^\circ,n) \to \CM_{X^\circ}(n)
\]
is again smooth, being the pullback of the smooth morphism $q_{F,n}\colon \mathrm{Q}_{F,n} \to \CM_X(n)$ along the open immersion $\CM_{X^\circ}(n) \into \CM_X(n)$. This applies for instance to $(X,F,X^\circ) = (\BP^3,\OO^{\oplus r},\BA^3)$. Thus we get the next corollary.
\end{remark}

\begin{corollary}\label{smoothness_forgetting_framings}
The forgetful map $q_{r,n}\colon \Quot_{\BA^3}(\OO^{\oplus r},n) \to \CM_n$ is smooth of relative dimension $rn$.
\end{corollary}

Note that, for fixed $[E] \in \CM_n$, the choice of a surjection $\OO^{\oplus r}\onto E$ is nothing but the datum of $r$ (general enough) sections $\sigma_1,\ldots,\sigma_r \in \Hom(\OO,E) = \HH^0(E) = \BC^n$. So the fibre of $q_{r,n}$ over $[E]$ is an open subset of $\BC^{rn}$.

%%%%%%%%%%%%%%%%%%%%%%%%%%%%%%%%%%%%%%%%%%%%%%%%
%%%%%%%%%%%%%%%%%%%%%%%%%%%%%%%%%%%%%%%%%%%%%%%%
\section{Proof of Theorem \ref{main_thm}}\label{sec:proof_main_result}

In this section we prove \Cref{main_thm} (see \Cref{main_thm_BODY}) granting the (fundamental) auxiliary result \Cref{thm:comparison_d-crit_1}, which will be proved in \Cref{thm:comparison_d-crit}.

The moduli stack $\CM_n$ of $0$-dimensional coherent sheaves on length $n$ on $\BA^3$ can be seen as a stack of representations of the Jacobi algebra of a quiver with potential. Indeed, consider the $3$-loop quiver $L_3$ (\Cref{fig:3loop_quiver}), equipped with the potential $W=A[B,C]$.

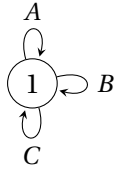
\begin{figure}[ht]
\centering
\begin{tikzpicture}[>=stealth,->,shorten >=2pt,looseness=.5,auto]
  \matrix [matrix of math nodes,
           column sep={3cm,between origins},
           row sep={3cm,between origins},
           nodes={circle, draw, minimum size=5.5mm}]
{ 
|(B)| 1 \\         
};
\tikzstyle{every node}=[font=\small\itshape]
\path[->] (B) edge [loop above] node {$A$} ()
              edge [loop right] node {$B$} ()
              edge [loop below] node {$C$} ();
\end{tikzpicture}
\caption{The $3$-loop quiver $L_3$.}\label{fig:3loop_quiver}
\end{figure}

\begin{notation}
Fix a quiver $Q = (Q_0,Q_1,s,t)$, where $Q_0$ is the vertex set, $Q_1$ is the edge set, $s$ and $t$ are the source and target maps $Q_1 \to Q_0$ respectively; for a dimension vector $\boldit{d} = (\boldit{d}_i)_i \in \BN^{Q_0}$, we let $\Rep_{\boldit{d}}(Q)$ denote the space of $\boldit{d}$-dimensional representations of $Q$, namely the affine space $\prod_{a \in Q_1}\Hom_{\BC}(\BC^{\boldit{d}_{s(a)}},\BC^{\boldit{d}_{t(a)}})$. The group $\GL_{\boldit{d}}=\prod_{i\in Q_0}\GL_{\boldit{d}_i}$ acts on $\Rep_{\boldit{d}}(Q)$ by conjugation; the moduli stack of $\boldit{d}$-dimensional representations is the quotient stack $\mathfrak M_{\boldit{d}}(Q) = [\Rep_{\boldit{d}}(Q)/\GL_{\boldit{d}}]$.
\end{notation}

The critical locus of the $\GL_n$-invariant regular function 
\begin{equation} \label{tr W}
f_n = (\Tr W)_n\colon \Rep_n(L_3) \to \BC, \qquad (A,B,C) \mapsto \Tr A[B,C],
\end{equation}
defines a closed $\GL_n$-invariant subscheme $U_n$ of the $3n^2$-dimensional affine space $\Rep_n(L_3) \cong \End_{\BC}(\BC^n)^{3}$ parametrising triples of \emph{pairwise commuting} endomorphisms. There is an isomorphism of Artin stacks
\begin{equation}\label{critical_Mn}
\begin{tikzcd}
\iota_n\colon [U_n / \GL_n] \arrow{r}{\sim} & \CM_n
\end{tikzcd}
\end{equation}
defined on closed points by sending the $\GL_n$-orbit $[A,B,C]$ of a triple $(A,B,C) \in U_n$ to the point $[E]$ represented by the direct sum $E=\OO_{p_1}\oplus \cdots \oplus \OO_{p_n}$, where the points $p_i$ are not necessarily distinct, and are determined by the diagonal entries of the matrices $A$ (for the $x$-coordinate), $B$ (for the $y$-coordinate) and $C$ (for the $z$-coordinate). Indeed, the matrices can be simultaneously put in upper triangular form precisely because they pairwise commute, and we are working modulo $\GL_n$. 
The source of the isomorphism $\iota_n$ is equipped with the algebraic d-critical structure
\[
s_n^{\crit} = f_n + (\dd f_n)^2 \in \HH^0\left(\CS^0_{[U_n / \GL_n]}\right)=\HH^0\left(\CS^0_{U_n}\right)^{\GL_n}.
\]
On the other hand, Brav and Dyckerhoff \cite{Brav-Dyckerhoff-II} proved that there is a $-1$-shifted symplectic structure $\omega_n$ on the derived Artin stack $\BCal{M}_n$ of $0$-dimensional coherent sheaves of length $n$ on $\BA^3$ (see Subsection~\ref{derived_stack_M_n} for details). We denote by
\begin{equation} \label{s_n der}
   (\CM_n,s_n^{\der}) = \tau(\BCal{M}_n,\omega_n) 
\end{equation} 
its truncation, defined through the `truncation functor' $\tau$ recalled in \Cref{thm:truncation}.

The following result, granted for now, will be proved in \Cref{subsec:d-crit-agree}.
\begin{theorem}
\label{thm:comparison_d-crit_1}
The isomorphism \eqref{critical_Mn} induces an identity 
\[
\iota_n^\ast s_n^{\der} = s_n^{\crit}
\]
of algebraic d-critical structures.
\end{theorem}

Now, for a fixed integer $r\geq 1$, let us consider the Quot scheme $\mathrm{Q}_{r,n} = \Quot_{\BA^3}(\OO^{\oplus r},n)$. As we now recall, this Quot scheme is a global critical locus, or, in other words, it admits a d-critical structure consisting of a single d-critical chart. This result in the case $r=1$ is due to Szendr\H{o}i \cite[Theorem 1.3.1]{MR2403807}. The `$r$-framed $3$-loop quiver' $\widetilde{L}_3$ (\Cref{fig:3loop_framed}) plays a crucial role.

\begin{figure}[ht]
\centering
\begin{tikzpicture}[>=stealth,->,shorten >=2pt,looseness=.5,auto]
  \matrix [matrix of math nodes,
           column sep={3cm,between origins},
           row sep={3cm,between origins},
           nodes={circle, draw, minimum size=7.5mm}]
{ 
|(A)| \infty & |(B)| 1 \\         
};
\tikzstyle{every node}=[font=\small\itshape]
\path[->] (B) edge [loop above] node {$A$} ()
              edge [loop right] node {$B$} ()
              edge [loop below] node {$C$} ();

\node [anchor=west,right] at (-0.15,0.11) {$\vdots$};

\draw (A) to [bend left=25,looseness=1] (B) node [midway,above] {};
\draw (A) to [bend left=40,looseness=1] (B) node [midway] {};
\draw (A) to [bend right=35,looseness=1] (B) node [midway,below] {};
\end{tikzpicture}
\caption{The $r$-framed $3$-loop quiver $\widetilde{L}_3$.}\label{fig:3loop_framed}
\end{figure}
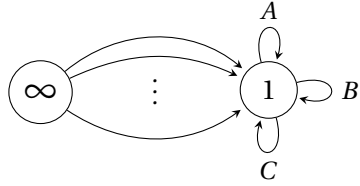

\begin{theorem}[{\cite[Theorem 2.6]{BR18}}]
\label{critical_quot}
Fix a vector $\theta = (\theta_1,\theta_2)\in \BR^2$, with $\theta_1>\theta_2$. Let 
\[
\NCQuot^n_r = \Rep_{(1,n)}^{\theta-\mathrm{st}}(\widetilde{L}_3) / \GL_n 
\]
be the moduli space of $\theta$-stable representations of the $r$-framed $3$-loop quiver $\widetilde L_3$, and consider the regular function $f_{r,n}\colon \NCQuot^n_r \to \BA^1$ induced by the potential $W=A[B,C]$. Then there is an isomorphism 
\[
\begin{tikzcd}
\iota_{r,n}\colon \crit(f_{r,n}) \arrow{r}{\sim} & \mathrm{Q}_{r,n}.
\end{tikzcd}
\]
\end{theorem}

The main result of this paper (\Cref{main_thm}, proved in \Cref{main_thm_BODY} below) states that the d-critical structure 
\[
s_{r,n}^{\crit} = f_{r,n} + (\dd f_{r,n})^2 \in \HH^0\left(\CS^0_{\crit(f_{r,n})}\right)
\]
determined by the function $f_{r,n}$ in \Cref{critical_quot} agrees, up to the isomorphism $\iota_{r,n}$, with the pullback of $s_n^{\der}$ along the smooth forgetful morphism $q_{r,n}$.

\begin{remark}
According to our conventions, in a dimension vector (or a stability condition) on a framed quiver $\infty \to Q$ the first entry always refers to the framing vertex $\infty$.  As explained in \cite{BR18,cazzaniga2020higher}, stability of a representation $(A,B,C,v_1,\ldots,v_r) \in \Rep_{(1,n)}(\widetilde L_3)$ with respect to $\theta$ translates into the condition that the framing vectors $v_1,\ldots,v_r \in \Hom_{\BC}(V_\infty,V_1) = \BC^n$ generate the underlying (unframed) representation $(A,B,C) \in \Rep_n(L_3)$. Since $\GL_n$ acts freely on $\Rep_{(1,n)}^{\theta-\mathrm{st}}(\widetilde{L}_3)$, this exhibits $\NCQuot^n_r$ as a \emph{smooth} quasiprojective variety of dimension $2n^2+rn$.
\end{remark}

We have a cartesian diagram 
\begin{equation}
\label{diagram_main_players}
\begin{tikzcd}[row sep = large,column sep=large]
\crit(f_{r,n})\MySymb{dr}\arrow[swap]{d}{\widetilde q_{r,n}}\arrow[r,"\iota_{r,n}","\sim"'] & \mathrm{Q}_{r,n}\arrow{d}{q_{r,n}} \\
{[U_n/\GL_n]}\arrow[r,"\iota_n","\sim"'] & \CM_n
\end{tikzcd}    
\end{equation}
where $\iota_{r,n}$ is the isomorphism of \Cref{critical_quot} and $\iota_n$ is the isomorphism \eqref{critical_Mn}.

\begin{prop}\label{prop:PB_of_sn}
There is an identity of d-critical structures
\[
\widetilde q_{r,n}^\ast s_n^{\crit} = s_{r,n}^{\crit} \in \HH^0\left(\CS^0_{\crit(f_{r,n})}\right),
\]
where $s_{r,n}^{\crit} = f_{r,n} + (\dd f_{r,n})^2$ is the d-critical structure defined by the function $f_{r,n}$.
\end{prop}

\begin{proof}
Forgetting the framing data yields a smooth morphism
\[
p_{r,n}\colon \NCQuot^n_r \to \mathfrak{M}_n(L_3) = [\Rep_n(L_3)/\GL_n]
\]
fitting in a cartesian diagram
\[
\begin{tikzcd}[row sep=large]
\crit(f_{r,n})\MySymb{dr} \arrow[swap]{d}{\widetilde{q}_{r,n}}\arrow[hook]{r}  & \NCQuot^n_r \arrow{d}{p_{r,n}} \\
{[}U_n/\GL_n{]}\arrow[hook]{r} & \mathfrak{M}_n(L_3)
\end{tikzcd}
\]
where the horizontal arrows are closed immersions. The function $f_n$ introduced in \eqref{tr W} descends to a function $\mathfrak{M}_n(L_3) \to \BA^1$, still denoted $f_n$, with critical locus $[U_n/\GL_n]$.
The conclusion then follows by observing that $p_{r,n}^\ast f_n = f_{r,n}$, which is true because $f_{r,n}$ does not interact with the framing data.
\end{proof}

We can now complete the proof of \Cref{main_thm}.

\begin{theorem}\label{main_thm_BODY}
Let $s_{r,n}^{\der}\in \HH^0(\CS^0_{\mathrm{Q}_{r,n}})$ be the pullback of the d-critical structure $s_n^{\der}$ (defined in \eqref{s_n der} by truncating $\omega_n$) along $q_{r,n}$. Then the isomorphism $\iota_{r,n}$ of \Cref{critical_quot} induces an identity
\[
\iota_{r,n}^\ast s_{r,n}^{\der} = s_{r,n}^{\crit} \in \HH^0\left(\CS^0_{\crit(f_{r,n})}\right).
\]
\end{theorem}

\begin{proof}
We have
\begin{align*}
    \iota_{r,n}^\ast s_{r,n}^{\der}
   &\,=\, \iota_{r,n}^\ast q_{r,n}^\ast s_n^{\der} & \textrm{by definition} 
   \\
    &\,=\, \widetilde q_{r,n}^\ast \iota_n^\ast s_n^{\der}& \textrm{by Diagram \eqref{diagram_main_players}} \\
    &\,=\,\widetilde q_{r,n}^\ast s_n^{\crit} & \textrm{by \Cref{thm:comparison_d-crit_1}}\\
    &\,=\,s_{r,n}^{\crit} & \textrm{by \Cref{prop:PB_of_sn}},
\end{align*}
which concludes the proof.
\end{proof}

%%%%%%%%%%%%%%%%%%%%%%%%%%%%%%%%%%%%%%%%%%%%%%%%
%%%%%%%%%%%%%%%%%%%%%%%%%%%%%%%%%%%%%%%%%%%%%%%%
\section{The d-critical structure(s) on \texorpdfstring{$\CM_n$}{}}

In this section we compare the two d-critical structures
\[
s_n^{\crit} \in \HH^0\left(\CS^0_{[U_n/\GL_n]}\right),\quad s_n^{\der} \in \HH^0\left(\CS^0_{\CM_n}\right)
\]
on the isomorphic spaces $[U_n/\GL_n]$ and $\CM_n$, introduced in \Cref{sec:proof_main_result}. These d-critical structures come from quiver representations and symplectic derived algebraic geometry respectively. Our goal (achieved in \cref{thm:comparison_d-crit}) is to prove \Cref{thm:comparison_d-crit_1}.

To do so, we first analyse and give explicit local d-critical charts for the two d-critical structures. For $s_n^{\crit}$, we use the geometry of the quotient stack $[U_n / \GL_n]$ and Luna's \'{e}tale slice theorem. For $s_n^{\der}$, we take advantage of the derived deformation theory of the $-1$-shifted symplectic stack $\BCal{M}_n$ to obtain explicit formal charts in Darboux form (cf.~\cite{BBBJ}). Finally, we argue that one can pass from formal neighbourhoods to honest smooth neighbourhoods to obtain the equality of the d-critical structures.

\begin{notation}\label{notation:matrices}
Throughout we shall use the notation $\Mat_{a,b}(R)$ to denote the space of matrices with $a$ rows and $b$ columns with entries in a ring $R$. If $R=\BC$, we simply write $\Mat_{a,b}$.
\end{notation}

%%%%%%%%%%%%%%%%%%%%%%%%%%%%%%%%%%%%%%%%%%%%%%%%%%%%%%%%%%%%%%%%
\subsection{The d-critical structure \texorpdfstring{$s_n^{\crit}$}{} coming from quiver representations} \label{subsection: s_n^crit}

Let $[E] \in \CM_n$ be a closed point. Then the corresponding sheaf $E$ must be polystable, i.e.~of the form 
\begin{equation} \label{polystable sheaf}
    E = \bigoplus_{i=1}^k \BC^{a_i} \otimes \OO_{p_i}
\end{equation}
where, for $i=1,\ldots,k$, the points $p_i = (\alpha_i, \beta_i, \gamma_i) \in \BA^3$ are pairwise distinct and $a_i$ are positive integers such that $a_1+\cdots+a_k=n$.

We fix some notation for convenience. 
Let $Y_n = \Rep_n(L_3) = \End_{\BC}(\BC^n)^3$ be the vector space of triples of $n \times n$ matrices, and let $f_n\colon Y_n \to \BC$ be the regular function \eqref{tr W}.
Denote by ${\boldit{a}} = (a_1,\ldots, a_k)$ the $k$-tuple of positive integers  determined by \eqref{polystable sheaf}. Form the product
\[
Y_{\boldit{a}} = \prod_{i=1}^k Y_{a_i}.
\]
Then there is a closed embedding $\Phi_{\boldit{a}} \colon Y_{\boldit{a}} \into Y_n$ by block diagonal matrices with square diagonal blocks of sizes $a_1,\ldots, a_k$. The reductive algebraic group $\GL_{\boldit{a}} = \prod_{i=1}^k \GL_{a_i}$ acts on $Y_{\boldit{a}}$ by componentwise conjugation and $\Phi_{\boldit{a}}$ is equivariant with respect to the inclusion $\GL_{\boldit{a}} \subset \GL_n$ by block diagonal matrices of the same kind.

On the space $Y_{\boldit{a}}$ we have the $\GL_{\boldit{a}}$-invariant potential
\begin{equation} \label{def of g_a}
g_{\boldit{a}} = f_{a_1}\oplus \cdots\oplus f_{a_k} \colon Y_{\boldit{a}} \to \BA^1, \quad (A_i,B_i,C_i)_i \mapsto \sum_{i=1}^k \Tr A_i [B_i, C_i],
\end{equation}
where $(A_i, B_i, C_i) \in Y_{a_i}$ for $i= 1,\ldots, k$. 
We have the obvious relation $f_n\circ \Phi_{\boldit{a}} = g_{\boldit{a}}$.
If we set 
\[
U_n = \crit(f_n) \subset Y_n,\quad
U_{\boldit{a}} = \crit(g_{\boldit{a}}) \subset Y_{\boldit{a}},
\]
the restriction $\Phi_{\boldit{a}}\colon U_{\boldit{a}} \into U_n$ is still equivariant with respect to the inclusion $\GL_{\boldit{a}} \subset \GL_n$, and so it induces a morphism of algebraic stacks
\begin{equation}
    \psi_{\boldit{a}} \colon [U_{\boldit{a}}/\GL_{\boldit{a}}] \to [U_n/\GL_n].
\end{equation}
Note that the identification $\crit(g_{\boldit{a}}) = \prod_i \crit(f_{a_i})= \prod_iU_{a_i}$ combined with the product of the isomorphisms $\iota_{a_i}$ as in \eqref{critical_Mn} induces a canonical identification
\[
\begin{tikzcd}
{[}U_{\boldit{a}}/\GL_{\boldit{a}}{]} = \displaystyle\prod_{i=1}^k\,{[}U_{a_i}/\GL_{a_i}{]} \arrow{r}{\sim} & \displaystyle\prod_{i=1}^k \CM_{a_i} = \CM_{\boldit{a}}.
\end{tikzcd}
\]
We can identify the map $\psi_{\boldit{a}}$ with the direct sum map
\[
\psi_{\boldit{a}} \colon \CM_{\boldit{a}} \to \CM_n,
\]
denoted the same way,
taking a $B$-valued point of $\CM_{\boldit{a}}$, i.e.~a $k$-tuple $(\mathscr E_1,\ldots,\mathscr E_k)$ of $B$-flat families of $0$-dimensional sheaves $\mathscr E_i \in \Coh(B\times \BA^3)$, to their direct sum $\bigoplus_{1\leq i\leq k} \mathscr E_i$. (The direct sum of flat sheaves if flat by \cite[\href{https://stacks.math.columbia.edu/tag/05NC}{Tag 05NC}]{stacks-project}).

\begin{lemma}\label{Lemma:etaleness_psi}
Let $\CX_{\boldit{a}} \subset \CM_{\boldit{a}}$ be the open substack parametrising $k$-tuples of sheaves with pairwise disjoint support. Then the morphism 
\[
\psi_{\boldit{a}}\big|_{\CX_{\boldit{a}}}\colon \CX_{\boldit{a}} \to \CM_n
\]
is \'etale.
\end{lemma}

\begin{proof}
According to \cite[\href{https://stacks.math.columbia.edu/tag/0CIK}{Tag 0CIK}]{stacks-project}, to show that $\psi \defeq \psi_{\boldit{a}}|_{\CX_{\boldit{a}}}$ is \'etale it is enough to find an algebraic space $W$, a faithfully flat morphism $\rho\colon W \to \CM_n$ locally of finite presentation and an \'etale morphism $T \to W\times_{\rho,\CM_n,\psi}\CX_{\boldit{a}}$ such that $T \to W$ is \'etale. We start by picking $W=\Quot_{\BA^3}(\OO^{\oplus n},n)$ along with the smooth (cf.~\Cref{smoothness_forgetting_framings}) morphism
\[
\rho = q_{n,n}\colon \Quot_{\BA^3}(\OO^{\oplus n},n) \to \CM_n
\]
sending, for any test scheme $B$, a $B$-flat quotient $\OO_{B\times \BA^3}^{\oplus n} \onto \CT$ to the object $\CT \in \CM_n(B)$.
Note that $\rho$ is surjective in the sense of \cite[\href{https://stacks.math.columbia.edu/tag/04ZR}{Tag 04ZR}]{stacks-project}. Indeed, for any closed point $[E]$, we have a direct sum decomposition as in \eqref{polystable sheaf}, thus the sheaf $E$ receives a surjection from $\OO^{\oplus n}$, which is nothing but the direct sum of the surjections $\OO^{\oplus a_i} \onto \BC^{a_i}\otimes \OO_{p_i}$ --- their direct sum is again surjective because $p_i\neq p_j$ for all $i\neq j$.

Next, we form the cartesian diagram
\[
\begin{tikzcd}[row sep=large,column sep=large]
T\MySymb{dr}\arrow[swap]{d}{\mu}\arrow{r}  & \CX_{\boldit{a}} \arrow{d}{\psi} \\
\Quot_{\BA^3}(\OO^{\oplus n},n) \arrow[swap]{r}{\rho} & \CM_n
\end{tikzcd}
\]
and we pick the identity $T = \Quot_{\BA^3}(\OO^{\oplus n},n)\times_{\rho,\CM_n,\psi}\CX_{\boldit{a}}$ as an \'etale map. We need to show that $\mu$ is \'etale. But this follows from \cite[Proposition A.3]{BR18}, after observing that $T$ is the open subscheme of
\[
\prod_{i=1}^k\Quot_{\BA^3}(\OO^{\oplus n},a_i)
\]
parametrising $k$-tuples of quotients with pairwise disjoint support, and $\mu$ is nothing but the map taking a tuple of surjections to their direct sum.
\end{proof}

Let $E$ be a polystable sheaf supported on points $p_i = (\alpha_i,\beta_i,\gamma_i)$ for $1\leq i\leq k$, as in \eqref{polystable sheaf}.
The point $[E] \in \CM_n$ corresponds under $\iota_n$ to the orbit of the triple of matrices $y_E = \Phi_{\boldit{a}}( v_E ) \in U_n\subset Y_n$, where $v_E = (\unal, \unbet, \ungam) \in U_{\boldit{a}} \subset Y_{\boldit{a}}$ is given by
\begin{equation}
\unal = (\alpha_1 \Id_{a_1},\ldots, \alpha_k \Id_{a_k}),\  \unbet = (\beta_1 \Id_{a_1},\ldots, \beta_k \Id_{a_k}),\ 
\ungam = (\gamma_1 \Id_{a_1},\ldots, \gamma_k \Id_{a_k}).
\end{equation}
We let $\overline v_E$ denote the image of $v_E$ along the smooth atlas $U_{\boldit{a}} \to [U_{\boldit{a}}/\GL_{\boldit{a}}]$, and similarly for $\overline y_E \in [U_n/\GL_n]$. The morphism $\psi_{\boldit{a}}$ maps $\overline v_E \in [U_{\boldit{a}} / \GL_{\boldit{a}}]$ to $\overline y_E \in [U_n/\GL_n]$.
For convenience, we will identify $\unal, \unbet, \ungam$ with their images in $Y_n$ under $\Phi_{\boldit{a}}$, considering them as $n \times n$ block diagonal matrices and more generally we identify elements of $Y_{\boldit{a}}$ with their image in $Y_n$ under $\Phi_{\boldit{a}}$.

\begin{lemma} \label{etale d-critical chart s_n^crit}
There is an open $\GL_{\boldit{a}}$-invariant neighbourhood $v_E \in V \subset Y_{\boldit{a}}$ such that, if $U = \crit(g_{\boldit{a}}|_{V})$, the morphism $\psi_E \defeq \psi_{\boldit{a}}|_{[U/\GL_{\boldit{a}}]}$ is \'{e}tale. In particular $\psi_{\boldit{a}}$ is \'etale at $\overline v_E$.
Moreover, if $\phi_E$ denotes the composition
\[
\begin{tikzcd}
\phi_E \colon U \arrow{r} &
{[}U/\GL_{\boldit{a}}{]} \arrow[hook]{r} & 
{[}U_{\boldit{a}}/\GL_{\boldit{a}}{]} \arrow{r}{\psi_{\boldit{a}}} &
{[}U_n/\GL_n{]},
\end{tikzcd}
\]
then we have an identity of d-critical structures
\[
\phi_E^* s_n^{\crit} = g_{\boldit{a}}|_{V} + (\dd g_{\boldit{a}}|_{V})^2 \in \HH^0(\CS_{U}^0)^{\GL_{\boldit{a}}} \subset \HH^0(\CS_{U}^0).
\]
\end{lemma}

\begin{proof}
The first statement follows from \cref{Lemma:etaleness_psi}, up to replacing the open substack $[U/\GL_{\boldit{a}}]\into [U_{\boldit{a}}/\GL_{\boldit{a}}]$ with its intersection with $\CX_{\boldit{a}}$. We now compare $f_n$ and $g_{\boldit{a}}$.

For a complex matrix $A \in \Mat_{n,n}$ (cf.~\Cref{notation:matrices}), we consider it as a block matrix with diagonal blocks $A_{11}, \ldots, A_{kk}$ of sizes $a_1 \times a_1,\ldots, a_k \times a_k$ and off-diagonal blocks $A_{ij}$ of sizes $a_i \times a_j$. 

Let $Y_n = Y_{\boldit{a}} \oplus Y_n^\perp$ be the direct sum decomposition where $Y_n^\perp $ is the subspace of triples of matrices $(A,B,C)$ whose diagonal blocks are zero. This decomposition is $\GL_{\boldit{a}}$-invariant, since the conjugation action of $\GL_{\boldit{a}}$ is given on blocks by the formula 
\begin{equation*} \label{gl_a action on blocks}
(h \cdot (A,B,C))_{ij} = \left( h_{ii} A_{ij} h_{jj}^{-1}, h_{ii} B_{ij} h_{jj}^{-1}, h_{ii} C_{ij} h_{jj}^{-1} \right).
\end{equation*}

The derivative $\gl_n \to T_{y_E} Y_n$ of the $\GL_n$-action on $Y_n$ at the point $y_E = (\unal, \unbet, \ungam)$ is given by the map 
\begin{align*}
\sigma \colon & \gl_n = \Mat_{n,n} \lr T_{y_E} Y_n \simeq Y_n \\
 & X \mapsto \left( [X, \unal],\ [X, \unbet],\ [X, \ungam] \right)
\end{align*}
Notice that the $ij$-th blocks of the triple of matrices $\sigma(X)$ are given by the triple of $a_i \times a_j$ matrices
\begin{equation} \label{formula for sigma on block ij}
\sigma(X)_{ij} = -\left( (\alpha_i - \alpha_j) X_{ij}, (\beta_i - \beta_j) X_{ij}, (\gamma_i - \gamma_j) X_{ij}\right).
\end{equation}
In particular, $\mathrm{im} (\sigma)$ is a $\GL_{\boldit{a}}$-invariant subspace of $Y_n^\perp$. 

Define a subspace $\mathrm{im}(\sigma)^\perp \subset Y_n$ by the condition that $(X,Y,Z) \in \mathrm{im}(\sigma)^\perp$ if and only if for all indices $i \neq j$ we have
\begin{equation} \label{def of im(sigma) complement}
\overline{(\alpha_i - \alpha_j)} X_{ij} + \overline{(\beta_i - \beta_j)} Y_{ij} + \overline{(\gamma_i - \gamma_j)} Z_{ij} = 0.
\end{equation}

We thus have a $\GL_{\boldit{a}}$-invariant direct sum decomposition
\begin{equation} \label{decomposition of tangent space}
    T_{y_E} Y_n \simeq Y_n = Y_{\boldit{a}} \oplus \mathrm{im}(\sigma) \oplus Y_n^\sigma
\end{equation}
where $Y_{\boldit{a}} \oplus Y_n^\sigma = \mathrm{im}(\sigma)^\perp$ and $Y_n^\sigma$ is the $\GL_{\boldit{a}}$-invariant complement of $Y_{\boldit{a}}$ in $\mathrm{im}(\sigma)^\perp$ consisting of the matrices that have zero diagonal blocks. Since $\mathrm{im}(\sigma) = T_{y_E} (\GL_n \cdot y_E)$, by Luna's \'{e}tale slice theorem there exists a sufficiently small $\GL_{\boldit{a}}$-invariant open neighbourhood $S$ of $y_E \in Y_{\boldit{a}} \oplus Y_n^\sigma$ which is an \'{e}tale slice for the $\GL_n$-action on $Y_n$ at $y_E$.
Therefore, $T = U_n \times_{Y_n} S$ is an \'{e}tale slice for the $\GL_n$-action on $U_n$ at $y_E$. In particular, $T$ is cut out by the equations $\dd f_n|_{S} =0 $ in $S$ and thus $T = \crit(f_n|_{S}) \subset S$.

Since $f_n|_S = \Tr A[B,C]|_S$, it is easy to check that the derivatives of $f_n$ with respect to the coordinates of $Y_n^\sigma$ vanish on $Y_{\boldit{a}}$. Moreover, the Hessian $H$ of $f_n$ at the point $y_E \in Y_n$ is given by the quadratic form whose value at $(X,Y,Z) \in T_{y_E} Y_n \simeq Y_n$ is
\begin{equation}
    H(X,Y,Z) = \Tr \left( X[\unbet, Z] + X[Y,\ungam] + \unal[Y,Z] \right).
\end{equation}

By direct computation, the form $H$ is non-degenerate on $Y_n^\sigma$. To see this, working in terms of blocks, we have
\begin{align*}
H(X,Y,Z) & = \sum_{i \neq j} \Tr \left( (\beta_i - \beta_j) Z_{ij} X_{ji} - (\gamma_i - \gamma_j) Y_{ij} X_{ji} - (\alpha_i - \alpha_j) Z_{ij}Y_{ji} \right).
\end{align*}

For simplicity, we write each pair of summands corresponding to the indices $i \neq j$ suggestively in the form
\begin{equation} \label{simple summand}
    \beta \Tr(Z_1 X_2) - \gamma \Tr(Y_1 X_2) -\alpha \Tr(Z_1 Y_2) -\beta \Tr (Z_2 X_1) + \gamma \Tr(Y_2 X_1) + \alpha \Tr(Z_2 Y_1)
\end{equation}
where $X_1, Y_1, Z_1$ (resp.~$X_2, Y_2, Z_2$) stand for $X_{ij}, Y_{ij}, Z_{ij}$ (resp.~$X_{ji}, Y_{ji}, Z_{ji}$) and $\alpha, \beta, \gamma$ stand for $\alpha_i - \alpha_j, \beta_i - \beta_j, \gamma_i - \gamma_j$ respectively.

Then \eqref{def of im(sigma) complement} takes the equivalent form
\begin{align*}
\overline{\alpha} X_1 + \overline{\beta} Y_1 + \overline{\gamma} Z_1 &= 0  \\
\overline{\alpha} X_2 + \overline{\beta} Y_2 + \overline{\gamma} Z_2 &= 0.
\end{align*}

At least one of $\alpha, \beta, \gamma$ is non-zero, so without loss of generality we assume $\alpha \neq 0$, as the computation is analogous in the other two cases. Solving for $X_1$ and $X_2$, substituting into \eqref{simple summand} and simplifying, gives the expression
\[
\frac{1}{\overline{\alpha}} \left( |\alpha|^2 + |\beta|^2 + |\gamma|^2 \right) \Tr \left( Y_1 Z_2 - Y_2 Z_1 \right).
\]
This is non-degenerate and therefore $H$ must be non-degenerate.

Since $f_n|_{Y_{\boldit{a}}} = g_{\boldit{a}}$, the non-degeneracy of $H$ on $Y_n^\sigma$ implies that, after possibly replacing $S$ by a $\GL_{\boldit{a}}$-invariant open neighbourhood of $y_E$ in $S$, we may assume that the ideal $I = (\dd f_n|_{S})$ coincides with the ideal $(\dd g_{\boldit{a}}, K)$, where $K$ is the ideal of $Y_{\boldit{a}}$ in $S$ generated by the coordinates of $Y_n^\sigma$. 

Thus, letting $U = U_\bfa \times_S {Y_\bfa} $, we see that $U$ is an open neighbourhood of $v_E$ in $U_{\boldit{a}}$ and the \'{e}tale map
\[ 
\kappa = \psi_{\boldit{a}}\big|_{[U/\GL_{\boldit{a}}]} \colon [U/\GL_{\boldit{a}}] \lr [U_{\boldit{a}}/\GL_{\boldit{a}}] \lr [U_n/\GL_n]
\]
satisfies by construction
\[
\left( \kappa^* s_n^{\crit} \right)\big|_U = f_n\big|_{S} + I^2 \in \HH^0(\CS_{U}^0)^{\GL_{\boldit{a}}}.
\]

Let $V = Y_{\boldit{a}} \times_{S} Y_n^\sigma$. By definition, $V$ is a $\GL_{\boldit{a}}$-invariant open neighbourhood of $y_E$ in $Y_{\boldit{a}}$. We have the embedding $\Psi \colon V \to S$ such that $\Psi^* f_n|_{S} = g_{\boldit{a}}|_{V}$ and $U = \crit(g_{\boldit{a}}|_{V})$. Write $J$ for the ideal $(\dd g_{\boldit{a}}|_{V})$. Then, by the definition of the sheaf $\CS_{U}^0$, the commutative diagram
\[
\begin{tikzcd}
0 \arrow{r} 
& \CS^0_U \arrow[equal]{d} \arrow{r}
& \OO_{S}/I^2 \arrow{d}{\Psi^\ast} \arrow{r}{\dd}
& \Omega_{S} / I \cdot \Omega_{S} \arrow{d}{\Psi^\ast} \\
0 \arrow{r} 
& \CS^0_U \arrow{r}
& \OO_{V}/J^2 \arrow{r}{\dd}
& \Omega_{V} / J \cdot \Omega_{V} 
\end{tikzcd}
\]
implies that $f_n|_{S} + I^2 = g_{\boldit{a}}|_{V} + J^2 \in \HH^0(\CS_{U}^0)^{\GL_{\boldit{a}}}$, which is exactly what we want.
\end{proof}

%%%%%%%%%%%%%%%%%%%%%%%%%%%%%%%%%%%%%%%%%%%%%%%%%%%%%%%%%%%%%%%%%%%%%%%%%%%%%%%%%%%%%%%%%%%%%%%%
\subsection{The derived stack $\BCal{M}_n$ of $0$-dimensional coherent sheaves of length $n$ on $\BA^3$} \label{derived_stack_M_n}

Let $\BCal{M}_n$ be the derived moduli stack of $0$-dimensional coherent sheaves of length $n$ on $\BA^3$. Its classical truncation is the stack $\CM_n$. We now recall the definition of $\BCal{M}_n$ and give an explicit description as a derived quotient stack for the convenience of the reader. 

Let $A \in \cdga$. By \cite{ToenVaquie} or \cite[Example~2.7]{Brav-Dyckerhoff-II}, $\BCal{M}_n$ assigns to $A$ the mapping space
\begin{equation} \label{dgcat mapping space}
    \mathrm{Map}_{\mathrm{dgCat}} \left( \BC[x,y,z],\ \mathrm{Proj}_n(A) \right). 
\end{equation}
Here $\BC[x,y,z]$ is viewed as the dg-category with one object and morphisms concentrated in degree $0$ given by $\BC[x,y,z]$, whereas $\mathrm{Proj}_n(A)$ is the dg-category whose objects are rank $n$ projective $A$-modules and the morphism complexes are given by $\Hom^\bullet(E, F)$ for two objects $E,F$.

Since projective modules are locally trivial and, as differential graded algebras, we have $\Hom^\bullet(A^n, A^n) = \gl_n \otimes_\BC A$, we can further simplify \eqref{dgcat mapping space} (up to possible shrinking) into
\begin{equation} \label{dga mapping space}
    \mathrm{Map}_{\dga} \left( \BC[x,y,z],\ \gl_n \otimes_\BC A \right),
\end{equation}
up to conjugation by $\GL_{n}$.

To compute this mapping space with respect to the model structure of the category $\dga$, we need a resolution of the algebra $\BC[x,y,z]$ by a semi-free dg-algebra. This is provided by the following proposition.

\begin{prop}
Let $Q_3$ be the semi-free dg-algebra with generators $x,y,z$ in degree $0$, $x^\ast,y^\ast,z^\ast$ in degree  $-1$ and $w$ in degree $-2$ and differential $\delta$ satisfying
\begin{equation} \label{differential of semi-free}
\begin{split}
    & \delta x = \delta y = \delta z = 0 \\ 
    & \delta x^\ast = [y,z],\ \delta y^\ast = [z,x],\ \delta z^\ast = [x,y] \\ 
    & \delta w = [x,x^\ast] + [y,y^\ast] + [z,z^\ast]. 
\end{split}
\end{equation}
Then the natural morphism $Q_3 \to \BC[x,y,z]$ is a quasi-isomorphism.
\end{prop}

\begin{proof}
The claim follows from the fact that the $\BC$-algebra $\BC\braket{x,y,z}$ generated over $\BC$ freely by three variables $x,y,z$ is Calabi--Yau of dimension 3 and \cite[Theorem~5.3.1]{CYAlgebras}.
\end{proof}

It follows from the proposition that the mapping space \eqref{dga mapping space} is computed by
\begin{equation*}
    \Hom_{\dga}(Q_3,\ \gl_n \otimes_\BC A)
\end{equation*}
which, using the results of \cite{DRep}, is naturally isomorphic to
\begin{equation*}
    \Hom_{\cdga} \left( (Q_3)_n,\ A \right).
\end{equation*}
Here $(Q_3)_n$ is the commutative dg-algebra with generators corresponding to the matrix entries of the $n \times n$ matrices $X,Y,Z$ in degree $0$, $X^\ast, Y^\ast, Z^\ast$ in degree $-1$ and $W$ in degree $-2$ with differential determined by \eqref{differential of semi-free}. In particular we have sets of generators given by the matrix entries of
\begin{equation*}
    \Rep_n(L_3) \simeq \gl_n^{\oplus 3}, \quad \Rep_n(L_3) \simeq \gl_n^{\oplus 3}, \quad \gl_n
\end{equation*}
in degrees $0$, $-1$ and $-2$ respectively.

Quotienting by the conjugation action of $\GL_n$, we obtain the following proposition.

\begin{prop} There exists an isomorphism of derived Artin stacks
\begin{equation} \label{description of derived M_n}
\begin{tikzcd}
    \boldsymbol{\iota}_n \colon {\bigl[}\dSpec (Q_3)_n \big/ \GL_n{\bigr]} \arrow{r}{\sim} & \BCal{M}_n.
\end{tikzcd}
\end{equation}
\end{prop}

%%%%%%%%%%%%%%%%%%%%%%%%%%%%%%%%%%%%%%%%%%%%%%%%%%%%%%%%%%%%%%%%%%%%%%%%%%%%%%%%%%%%%%%%%%%%%%%%
\subsection{The d-critical structure \texorpdfstring{$s_n^{\der}$}{} coming from derived symplectic geometry} \label{sec:sn_derived}

Recall that Equation \eqref{s_n der} defines $s_n^{\der}$ to be the truncation of the $-1$-shifted symplectic structure $\omega_n$ on $\BCal{M}_n$ constructed by Brav and Dyckerhoff \cite{Brav-Dyckerhoff-II}. 

Let $[E] \in \CM_n$ be a closed point, so that the sheaf $E$ is polystable as in \eqref{polystable sheaf}. 
We now proceed to find formally smooth d-critical charts for $(\CM_n,s_n^{\der})$ around such points $[E]$ using the derived deformation theory of $\BCal{M}_n$, following \cite{Toda:Moduli_Ext}. In order to do this, we first need some preparations.

Let $p = (a_0, b_0, c_0) \in \BA^3$ be a point in coordinates $x_0, y_0, z_0$ for $\BA^3$. We denote the functions $x_0-a_0, y_0-b_0, z_0-c_0$ by $x, y, z$ respectively. For convenience, we also write $\OO$ in place of the ring $\OO_{\BA^3}=\BC[x_0,y_0,z_0]$ in what follows.

We have the Koszul resolution $Q_p^\bullet \to \OO_p$, given by the complex
\begin{equation}\label{eqn:Qp_complex}
    Q_p^\bullet \,=\, \bigl[Q^{-3} \to Q^{-2} \to Q^{-1} \to Q^0\bigr] \,=\, \bigl[ \OO \xrightarrow{A} \OO^{\oplus 3} \xrightarrow{B} \OO^{\oplus 3} \xrightarrow{C} \OO \bigr],
\end{equation}
where the morphisms $A, B, C$ are multiplications by the matrices (we are very slightly abusing notation here)
\begin{equation*}
    A = \begin{pmatrix} x \\ y \\ z \end{pmatrix}, \,\,\,\,\,
    B = \begin{pmatrix} 0 & -z & y \\ z & 0 & -x \\ -y & x & 0 \end{pmatrix}, \,\,\,\,\,
    C = \begin{pmatrix} x & y & z \end{pmatrix}.
\end{equation*}

Let 
\begin{equation} \label{g_p algebra}
    \fg_p = \Hom(Q_p^\bullet, Q_p^\bullet) = \RHom( \OO_p, \OO_p ).
\end{equation}

By composing morphisms in each degree, $\fg_p$ has the structure of an (infinite-dimensional)  dg-algebra $(\fg_p, \cdot, \delta )$ over $\BC$.
Moreover, let
\begin{equation} \label{min model of g_p algebra}
    \fg_p^{\min} = \bigoplus_{i=0}^3 \HH^i(\fg_p) = \bigoplus_{i=0}^3 \Ext^i(\OO_p, \OO_p).
\end{equation}
Using the Yoneda product 
\begin{equation*}
    m_2 \colon \Ext^i (\OO_p, \OO_p) \otimes  \Ext^j (\OO_p, \OO_p) \lr \Ext^{i+j} (\OO_p, \OO_p),
\end{equation*}
one can endow $\fg_p^{\min}$ with the structure of a graded commutative algebra over $\BC$ and thus of a dg-algebra $(\fg_p^{\min}, m_2, 0)$ with zero differential. As a graded algebra, it is clear that $\fg_p^{\min}$ is isomorphic to the exterior algebra $\Lambda^\bullet \BC^3$.

\begin{lemma}\label{lemma:dg}
There exists a morphism $I_p \colon (\fg_p^{\min}, m_2, 0) \to (\fg_p, \cdot, \delta)$ of dg-algebras, which induces an isomorphism on homology. In particular, the dg-algebra $\fg_p$ is formal as an $A_\infty$-algebra.
\end{lemma}

\begin{proof}
We construct the morphism $I_p$ explicitly. For brevity, we identify $\Hom_{\OO}(\OO^{\oplus s}, \OO^{\oplus t})$ with $t \times s$ matrices $\Mat_{t, s}(\OO)$ with entries in $\OO$.

The graded pieces of $\fg_p$ in degrees $0, \ldots, 4$ are
\begin{align*}
    \fg_p^{0} & = \bigoplus_i \Hom(Q^i, Q^{i}) \ \ \ \ = \OO \oplus \Mat_{3,3}(\OO) \oplus \Mat_{3,3}(\OO) \oplus \OO\\
    \fg_p^{1} & = \bigoplus_i \Hom(Q^i, Q^{i+1}) = \Mat_{3,1}(\OO) \oplus \Mat_{3,3}(\OO) \oplus \Mat_{1,3}(\OO)\\
    \fg_p^{2} & = \bigoplus_i \Hom(Q^i, Q^{i+2}) = \Mat_{3,1}(\OO) \oplus \Mat_{1,3}(\OO) &\\
    \fg_p^{3} & = \bigoplus_i \Hom(Q^i, Q^{i+3}) = \OO \\
    \fg_p^{4} & = \bigoplus_i \Hom(Q^i, Q^{i+4}) = 0.  
\end{align*}

Define elements $\hat{1} \in \fg_p^0$ and $\hat{x}, \hat{y}, \hat{z} \in \fg_p^{1}$ by
\begin{align*}
    \hat{1} & = \left( 1, \Id_3, \Id_3, 1 \right) \\
    \hat{x} & = \left( 
    \begin{pmatrix} 1 \\ 0 \\ 0 \end{pmatrix}, \ 
    \begin{pmatrix} 0 & 0 & 0 \\ 0 & 0 & -1 \\ 0 & 1 & 0 \end{pmatrix}, \
    \begin{pmatrix} 1 & 0 & 0 \end{pmatrix}
    \right)\\
    \hat{y} & = \left( 
    \begin{pmatrix} 0 \\ 1 \\ 0 \end{pmatrix}, \ 
    \begin{pmatrix} 0 & 0 & 1 \\ 0 & 0 & 0 \\ -1 & 0 & 0 \end{pmatrix}, \
    \begin{pmatrix} 0 & 1 & 0 \end{pmatrix}
    \right)\\
    \hat{z} & = \left( 
    \begin{pmatrix} 0 \\ 0 \\ 1 \end{pmatrix}, \ 
    \begin{pmatrix} 0 & -1 & 0 \\ 1 & 0 & 0 \\ 0 & 0 & 0 \end{pmatrix}, \
    \begin{pmatrix} 0 & 0 & 1 \end{pmatrix}
    \right).
\end{align*}

We may verify directly that these satisfy the relations
\begin{align*}
    \hat{1} \cdot \hat{x} = \hat{x},\,\,\,\, \hat{1} \cdot \hat{y} = \hat{y},\,\,\,\, \hat{1} \cdot \hat{z} = \hat{z} \\
    \hat{x} \cdot \hat{x} = 0,\,\,\,\, \hat{y} \cdot\hat{y} = 0, \,\,\,\, \hat{z} \cdot \hat{z} = 0,
\end{align*}
as well as
\begin{align*}
    \hat{x} \cdot \hat{y} = - \hat{y} \cdot \hat{x} = \left(  
    \begin{pmatrix} 0 \\ 0 \\ 1 \end{pmatrix},\ 
    \begin{pmatrix} 1 & 0 & 0 \end{pmatrix}
    \right) \\
    \hat{y} \cdot \hat{z} = - \hat{z} \cdot \hat{y} = \left(  
    \begin{pmatrix} 1 \\ 0 \\ 0 \end{pmatrix},\ 
    \begin{pmatrix} 1 & 0 & 0 \end{pmatrix}
    \right) \\
    \hat{z} \cdot \hat{x} = - \hat{x} \cdot \hat{z} = \left(  
    \begin{pmatrix} 0 \\ 1 \\ 0 \end{pmatrix},\ 
    \begin{pmatrix} 0 & 1 & 0 \end{pmatrix}
    \right)
\end{align*}
\begin{align*}
    \hat{x} \cdot \hat{y} \cdot \hat{z} &= 2.
\end{align*}

So the associative algebra $R$ generated by $\hat{1}, \hat{x}, \hat{y}, \hat{z}$ is a sub-algebra of $\fg_p$ isomorphic to $\Lambda^\bullet \BC^3$.
Moreover, all of the elements of $R$ are cocycles for the differential $\delta$ of $\fg_p$ and their images in homology give bases for the homology groups $\HH^i(\fg_p) = \Ext^i(\OO_p, \OO_p)$. By the definition of the Yoneda product $m_2$, it is immediate that $R$ is isomorphic to $\fg_p^{\min}$, giving a morphism of dg-algebras $I_p \colon \fg_p^{\min} \to \fg_p$ which, by construction, induces the identity on homology, as desired.
\end{proof}

\begin{remark}
The element $\hat{x}\in \fg_p^{1}$ is the evaluation of the triple $(A, B, C)$ at the point $(1,0,0) \in \BC^3$. Analogous statements hold for $\hat{y}, \hat{z}$ and all products of $\hat{x}, \hat{y}, \hat{z}$. This motivates their definition.
\end{remark}

Let now $E$ be a polystable sheaf as in \eqref{polystable sheaf} and write $Q_E^\bullet = \bigoplus_{i=1}^k \BC^{a_i} \otimes Q_{p_i}^\bullet$. We then define
\begin{align*}
    \fg_E = \Hom(Q_E^\bullet, Q_E^\bullet) = \bigoplus_{i,j=1}^k \Hom(Q_{p_i}^\bullet, Q_{p_j}^\bullet) \otimes \Mat_{a_j, a_i} = \RHom(E, E)
\end{align*}
and
\begin{align} \label{min model of g_E algebra}
    \fg_E^{\mathrm{ss}} & = \bigoplus_{i=1}^k \fg_{p_i} \otimes \Mat_{a_i, a_i} \nonumber\\
    \fg_E^{\min} & = \bigoplus_{i=1}^k \fg_{p_i}^{\min} \otimes \Mat_{a_i, a_i}.
\end{align}
Then $\fg_E$ is a dg-algebra with differential $\delta$ induced from that of each summand $\fg_{p_i}$ and $\fg_E^{\mathrm{ss}} \subset \fg_E$ is a dg-subalgebra.
In the same fashion $\fg_E^{\min}$ is a commutative graded algebra.
Since for any $i \neq j$ and any $\ell$ we have $\Ext^\ell(\OO_{p_i}, \OO_{p_j}) = 0$, the inclusion $\fg_E^{\mathrm{ss}} \subset \fg_E$ is a quasi-isomorphism and $\fg_E^{\min}$ is isomorphic as a dg-algebra to $(\Ext^*(E, E), m_2, 0)$.

\begin{corollary}
Let $E$ be a polystable sheaf as in \eqref{polystable sheaf}. Then there exists a morphism 
\[
I_E \colon (\fg_E^{\min}, m_2, 0) \to (\fg_E, \cdot, \delta)
\]
of dg-algebras, which induces an isomorphism on homology. Thus, the dg-algebra $\fg_E$ is formal as an $A_\infty$-algebra.
\end{corollary}

\begin{proof}
The map $I_E$ is given by the composition of the morphism 
$\bigoplus_{i=1}^k I_{p_i} \otimes \id_{\Mat_{a_i, a_i}}$
and the inclusion $\fg_E^{\mathrm{ss}} \subset \fg_E$.
\end{proof}

We fix some notation and useful facts before proceeding. 

Let $\mathsf S = \widehat{S}(\fg_E^{\min, >0}[1])^\vee$ be the dual of the abelianisation of the Bar construction of the dg-algebra $\fg_E^{\min, >0}$, which is the Chevalley--Eilenberg dg-algebra of the dg-Lie algebra associated to the commutative dg-algebra $\fg_E^{\min}$, with Lie bracket on each $\fg_{p_i}^{\min} \otimes \Mat_{a_i,a_i}$ determined by the equations
\begin{equation} \label{commutator relations}
\begin{split}
    \left[ (\hat{x}_i^\vee \otimes A_i) , (\hat{y}_i^\vee \otimes B_i) \right] & = \hat{z}_i^\vee \otimes [A_i,B_i] \\
    \left[ (\hat{y}_i^\vee \otimes B_i), (\hat{z}_i^\vee \otimes C_i) \right] & = \hat{x}_i^\vee \otimes [B_i,C_i] \\
    \left[ (\hat{z}_i^\vee \otimes C_i), (\hat{x}_i^\vee \otimes A_i) \right] & = \hat{y}_i^\vee \otimes [C_i,A_i], 
\end{split}
\end{equation}
where each $\fg_{p_i}^{\min} \cong \Lambda^\bullet \BC^3$ using the given basis $\lbrace \hat{x}_i, \hat{y}_i, \hat{z}_i \rbrace$. Here we have used the fact that the Lie bracket in the dg-Lie algebra structure of the derived tangent complex $\fg_E[1]$ of $\BCal{M}_n$ at $[E]$ is given by commutator of matrices (see \cite[Proposition~3.3]{Brav-Dyckerhoff-II} or \cite[Proposition~2.4.4]{TangentLieAlgebra}).

Thus, the definition \eqref{min model of g_E algebra} implies that $\mathsf S$ is generated by subalgebras $\mathsf S_i$, where each $\mathsf S_i$ is a negatively graded commutative dg-algebra and the differential $\epsilon$ in degree $-1$ is determined by the morphisms 
\begin{align} \label{differential of C-E algebra}
     \Lambda^2 \left( (\BC^3)^\vee \otimes \Mat_{a_i, a_i}^\vee \right) \xrightarrow{\epsilon_i} \Sym^\bullet \left( (\BC^3 )^\vee \otimes \Mat_{a_i, a_i}^\vee \right)
\end{align}
satisfying, due to the Lie bracket relations \eqref{commutator relations},
\begin{equation}
\begin{split}
    (\hat{x}_i^\vee \otimes A_i) \wedge (\hat{y}_i^\vee \otimes B_i) & \mapsto \hat{z}_i^\vee \otimes [A_i,B_i] \\
    (\hat{y}_i^\vee \otimes B_i) \wedge (\hat{z}_i^\vee \otimes C_i) & \mapsto \hat{x}_i^\vee \otimes [B_i,C_i] \\
    (\hat{z}_i^\vee \otimes C_i) \wedge (\hat{x}_i^\vee \otimes A_i) & \mapsto \hat{y}_i^\vee \otimes [C_i,A_i]. 
\end{split}
\end{equation}

Let $\hat{Y}_{a_i}$ be the formal completion of ${Y}_{a_i} = \End_{\BC}(\BC^{a_i})^3$ at the origin and let $\hat{f}_{a_i}$ be the formal function 
\begin{align*}
    \hat{f}_{a_i} \colon \hat{Y}_{a_i} \to \BC, \quad \hat{f}_{a_i} (A_i,B_i,C_i) = \Tr A_i[B_i,C_i].
\end{align*}
Similarly, let 
\begin{align} \label{eqn:def of g hat}
    \hat{g}_{\boldit{a}} = \hat{f}_{a_1}\oplus \cdots\oplus \hat{f}_{a_k} \colon \hat{Y}_{\boldit{a}} = \prod_{i=1}^k \hat{Y}_{a_i} \to \BC
\end{align}
be the function $\hat{g}_{\boldit{a}}(v_1,\ldots, v_k) = \hat{f}_{a_1}(v_1) + \cdots + \hat{f}_{a_k}(v_k)$.

It follows immediately by \eqref{differential of C-E algebra} that
\begin{equation}
    \Spec \HH^0(\mathsf S) = \prod_{i=1}^k \crit ( \hat{f}_{a_i} ) = \crit \left( \bigoplus_{i=1}^k \hat{f}_{a_i} \right) = \crit (\hat{g}_{\boldit{a}}) \subset \hat{Y}_\bfa.
\end{equation}

We also note that by \cite{Brav-Dyckerhoff-II}, the pairing induced from the $-1$-shifted symplectic structure $\omega_n$ on the shifted tangent complex $\fg_E^{\min}$ and its components $\fg_{p_i}^{\min}\otimes \Mat_{a_i, a_i}$ coincides with the Serre duality pairing. We denote this pairing by $\zeta_E$. %Up to possibly shifting $g_{\boldit{a}}$ by elements in $I^2 = (\dd g_{\boldit{a}})^2$, we may assume that this is an actual equality of pairings (and not just a homotopy). Since this does not change the induced section $\hat{g}_{\boldit{a}} + (\dd \hat{g}_{\boldit{a}})^2 \in \HH^0(\CS^0_{\hat{U}})$, we suppress it in what follows.

We may now prove the following key lemma.

\iffalse
Then by \cite[Appendix~A]{Toda:2017aa}, the morphism 
\begin{equation*}
    \boldsymbol{\Spec}\ \mathsf S \to \BCal{M}_n
\end{equation*}
gives a formal atlas for $\BCal{M}_n$ at the point $[E]$ which moreover is in Darboux form with respect to the $-1$-shifted symplectic structure $\omega_n$. In particular, the truncation
\begin{equation*} 
    \Spec \HH^0\left(\mathsf S\right) = \crit (\hat{g}_{\boldit{a}}) \to \CM_n
\end{equation*}
gives a hull for $\CM_n$ at $[E]$, which is a formal smooth d-critical chart with respect to the d-critical structure $s_n^{\der}$. We have established the following lemma.
\fi

\begin{lemma} \label{formal etale d-critical chart s_n^der}
Let $[E] \in \CM_n$ be a polystable sheaf of the form \eqref{polystable sheaf}, and set $\hat{U} = \crit(\hat{g}_{\boldit{a}})$, where $\hat{g}_{\boldit{a}} \colon \hat{Y}_{\boldit{a}} \to \BC$ is given by \eqref{eqn:def of g hat}. Then there exists a formally smooth morphism
\begin{equation} \label{eqn:psi_E hat}
\hat{\phi}_E \colon \hat{U} \to \CM_n,
\end{equation}
mapping $0 \in \hat{U}$ to $[E] \in \CM_n$, inducing an isomorphism at the level of tangent spaces  and an identity
\[
(\hat{\phi}_E)^\ast s_n^{\der} = \hat{g}_{\boldit{a}} + (\dd \hat{g}_{\boldit{a}})^2 \in \HH^0(\CS^0_{\hat{U}}).
\]
\end{lemma}

\begin{proof}
We continue to work formally locally at the point $[E] \in \BCal{M}_n$. Write $\lbrace \hat{y}_\ell \rbrace$ for the set of natural coordinates on $\hat{Y}_{\boldit{a}}$ induced by the equalities $Y_{a_i} = \End_{\BC}(\BC^{a_i})^3$.%, and $\fm$ for the maximal ideal generated by the $\hat{y}_\ell$, cutting out $0 \in \Spec \ {\mathsf S}^{0} = \hat{Y}_{\boldit{a}}$. 
Since ${\mathsf S}^{-1}$ is isomorphic to $T_{\hat{Y}_{\boldit{a}}}$ as an ${\mathsf S}^0$-module by definition, we have a natural basis $\lbrace \check{y}_\ell = \frac{\partial}{\partial \hat{y}_\ell} \rbrace$ of ${\mathsf S}^{-1}$, which is dual to the basis $\lbrace \dd_{\dR} \hat{y}_\ell \rbrace$. 

By the definition of $\mathsf S$ and properties of formal derived stacks and formal completions (see, for example, \cite{ShiftedPoisson} and \cite{TangentLieAlgebra}), there is a formally smooth morphism $\boldsymbol{\Spec} \ {\mathsf S} \to \BCal{M}_n$. Observe that, by construction, the morphism maps the closed point $0 \in \boldsymbol{\Spec} \ {\mathsf S}$ to $[E] \in \BCal{M}_n$ and is minimal at $0$, i.e. it induces an equivalence between the cotangent complexes of $\boldsymbol{\Spec} \ {\mathsf S}$ at $0$ and of $\BCal{M}_n$ at $[E]$ in non-positive degrees, and the restriction of the differential of the cotangent complex of $\boldsymbol{\Spec} \ {\mathsf S}$ is zero. 

Starting with this morphism as input, by \cite[Proposition~5.7]{BBJ}, it follows that the pullback of the closed $2$-form $\omega_n$ to $\boldsymbol{\Spec} \ {\mathsf S}$ can be brought to the form $(\omega_n^0, 0, 0, \ldots)$, where $\omega_n^0$ is a non-degenerate $-1$-shifted $2$-form that satisfies $\dd_{\dR} \omega_n^0 = \delta \omega_n^0 = 0$. More precisely, we can put this formally smooth chart in Darboux form, meaning that there exist $\Phi \in {\mathsf S}^0$ and $\phi \in \left( \Omega_{\mathsf S}^1 \right)^{-1}$, which satisfy the equations $\omega_n^0 = \dd_{\dR} \phi$, $\dd_{\dR} \Phi = \delta \phi$, and by \cite[Theorem~5.18]{BBJ}, we may take $\phi$ to be of the form $\phi = \sum_\ell \alpha_\ell \dd_{\dR} \hat{y}_\ell$, where $\alpha_\ell \in {{\mathsf S}^{-1}}$. The lemma will follow if we can show that $\Phi = \hat{g}_{\boldit{a}}$ and $\phi = \sum_\ell \check{y}_\ell \dd_{\dR} \hat{y}_\ell$, possibly after an automorphism of the cdga $\mathsf S$.

Now, $\omega_n^0 |_0 = \zeta_E = \sum_\ell \dd_{\dR} \check{y}_\ell|_0 \dd_{\dR} \hat{y}_\ell|_0$, so, by a graded application of \cite[Proposition~10.2, Lemma~10.3]{KontSoibFormal}, there exists an affine change of coordinates $\Psi$ for $\Spec \ {\mathsf S}^0$ such that with respect to these new coordinates, which we denote by $\hat{y}_\ell'$, we have $\omega_n^0 \otimes_{{\mathsf S}} {\mathsf S}^0 = \sum_\ell \dd_{\dR} \check{y}_\ell \dd_{\dR} \hat{y}_\ell'$.

Hence, still writing $\phi = \sum_\ell \alpha_\ell \dd_{\dR} \hat{y}_\ell'$ with respect to the new coordinates, we obtain $\omega_n^0 \otimes_{{\mathsf S}} {\mathsf S}^0 = \sum_\ell \dd_{\dR} \check{y}_\ell \dd_{\dR} \hat{y}_\ell' = \dd_{\dR} \phi \otimes_{{\mathsf S}} {\mathsf S}^0 = \sum_\ell \dd_{\dR} \alpha_\ell \dd_{\dR} \hat{y}_\ell'$, which shows that $\alpha_\ell = \check{y}_\ell$ for all values of the index $\ell$.

We now have a commutative diagram
\[
\begin{tikzcd}[row sep=large,column sep=large]
    {\mathsf S}^{-1}\arrow{r}{\sim} & T_{{\mathsf S}^0} \arrow[swap]{d}{(\dd \Psi)^{-1}}\arrow{r}{\dd \hat{g}_{\boldit{a}}^\vee} & {\mathsf S}^0\arrow{d}{\Psi} \\
    {\mathsf S}^{-1}\arrow{r}{\sim} & T_{{\mathsf S}^0}\arrow{r}{\dd (\hat{g}_{\boldit{a}} \circ \Psi)^\vee} & {\mathsf S}^0
\end{tikzcd}
\]
which extends to an automorphism of the cdga ${\mathsf S}$. By construction, the pullback of $\phi$ under this automorphism is equal to $\sum_\ell \check{y}_\ell' \dd_{\dR} \hat{y}_\ell'$, where $\check{y}_\ell'$ is the basis element of $ T_{{\mathsf S}^0}$ corresponding to the partial derivative $\frac{\partial}{\partial \hat{y}_\ell'}$.

In particular, after pulling back, we have equalities $\dd_{\dR} (\Phi \circ \Psi) = \delta \phi = \dd_{\dR} (\hat{g}_{\boldit{a}} \circ \Psi)$, so that $\Phi \circ \Psi = \hat{g}_{\boldit{a}} \circ \Psi$ and $\Phi = \hat{g}_{\boldit{a}}$, as desired.
\end{proof}

\begin{remark}
In a previous version of this paper, the claim that the natural morphism 
\begin{equation*}
    \boldsymbol{\Spec}\ \mathsf S \to \BCal{M}_n
\end{equation*}
gives a formal atlas for $\BCal{M}_n$ at the point $[E]$ which is \textit{already} in Darboux form with respect to the $-1$-shifted symplectic structure $\omega_n$ \iffalse In that case, the truncation
\begin{equation*} 
    \Spec \HH^0\left(\mathsf S\right) = \crit (\hat{g}_{\boldit{a}}) \to \CM_n
\end{equation*}
gives a hull for $\CM_n$ at $[E]$, which is a formal smooth d-critical chart with respect to the d-critical structure $s_n^{\der}$ and the lemma follows immediately.
\fi
was based on an appendix in the paper \cite{Toda:2017aa}. This has since been removed, so we have added an independent proof of the lemma to reflect this change.

The basic idea behind the proof is that a shifted symplectic form on a formal scheme can be taken to have constant coefficients after reparametrization. In particular, its value at the unique closed point should determine a Darboux form in a unique way. Hence, if the formal scheme is already in Darboux form and the restriction of the naturally associated symplectic structure at the point matches another given one, then the Darboux forms and symplectic structures must match as well. In this spirit, the claim is a simple, particular case of a general principle, whose proof we expect will appear in full form in the literature sooner rather than later.
\end{remark}

%%%%%%%%%%%%%%%%%%%%%%%%%%%%%%%%%%%%%%%%%%%%%%%%
\subsection{The two d-critical structures are equal} \label{subsec:d-crit-agree}

We now prove \Cref{thm:comparison_d-crit_1}, showing that the two d-critical structures studied in this section actually coincide.

\begin{theorem}
\label{thm:comparison_d-crit}
The isomorphism $\iota_n \colon [U_n/\GL_n] \simto \CM_n$ from \eqref{critical_Mn} induces an identity
\[
\iota_n^\ast s_n^{\der} = s_n^{\crit}
\]
of algebraic d-critical structures.
\end{theorem}

\begin{proof} 
Let $[E] \in \CM_n$ be a polystable sheaf, given by the expression \eqref{polystable sheaf}.
Applying Lemma~\ref{etale d-critical chart s_n^crit}, we have an \'{e}tale morphism $\psi_E \colon [U / \GL_{\boldit{a}}] \to [U_n / \GL_n]$ such that the pullback under the induced morphism $\phi_E \colon U \to [U / \GL_{\boldit{a}}] \to [U_n / \GL_n]$ satisfies 
\begin{equation} \label{eq crit}
\phi_E^* s_n^{\crit} = g_{\boldit{a}}|_{V} + (\dd g_{\boldit{a}}|_{V})^2 \in \HH^0(\CS_{U}^0)^{\GL_{\boldit{a}}} \subset \HH^0(\CS_{U}^0),
\end{equation}
where $V \subset Y_{\boldit{a}}$ is an open neighbourhood of the point $v_E \in Y_{\boldit{a}}$ corresponding to $[E]$ and $g_{\boldit{a}}$ is given by \eqref{def of g_a}.
By Lemma~\ref{formal etale d-critical chart s_n^der}, we have a formal atlas
\[
\hat{\phi}_E \colon \hat{U} \to \CM_n, \qquad 0 \mapsto [E],
\]
satisfying
\begin{equation} \label{eq der}
\hat{\phi}_E^\ast (s_n^{\der}) = \hat{g}_{\boldit{a}} + (\dd \hat{g}_{\boldit{a}})^2 \in \HH^0(\CS_{\hat{U}}^0).
\end{equation}

Let $j_m \colon U_{E,m} \to U$ be the $m$-th order thickening of $v_E \in U$. We get induced morphisms $\iota_n \circ \phi_E \circ j_m \colon U_{E,m} \to \CM_n$. Since $\hat{\phi}_E \colon \hat{U} \to \CM_n$ is formally smooth, this compatible system lifts via $\hat{\phi}_E$ to give rise to a morphism $\gamma \colon \hat{U} \to \hat{U}$ such that $\hat{\phi}_E \circ \gamma = \iota_n \circ \phi_E \circ \eta$, where $\eta \colon \hat{U} \to U$ is the formal completion morphism of $U$ at $v_E$. That is, we have a commutative diagram
\[
\begin{tikzcd}
\hat{U}\arrow[swap]{dr}{\eta}\arrow{r}{\gamma} & \hat{U}\arrow{drr}{\hat{\phi}_E} & & \\
& U\arrow[swap]{r}{\phi_E} & {[}U_n/\GL_n{]}\arrow[swap]{r}{\iota_n} & \CM_n
\end{tikzcd}
\]
where, since $\iota_n \circ \phi_E$ and $\eta$ are isomorphisms to first order at $v_E$ (i.e.~at the level of tangent spaces), the same is true for $\gamma$, so $\gamma$ is an isomorphism.

The inclusion $\hat{U} \subset \hat{V}$ is cut out by quadrics (the derivatives of the cubic function $g_{\boldit{a}}$) and thus $\gamma$ extends to an isomorphism $\gamma \colon \hat{V} \to \hat{V}$. Let $h = \hat{g}_{\boldit{a}} \circ \gamma$ so that obviously $(\dd h)=(\dd \hat{g}_{\boldit{a}})$. The commutativity of the diagram, along with \Cref{eq der}, implies that
\begin{equation*} 
    \eta^\ast \phi_E^\ast \iota_n^\ast (s_n^{\der}) = \gamma^* \hat{\phi}_E^\ast (s_n^{\der}) = h + (\dd h)^2 \in \HH^0(\CS^0_{\hat{U}}).
\end{equation*}

After possible further shrinking of $V$ around $v_E$, we have 
\begin{equation} \label{loc 2.25}
\phi_E^\ast \iota_n^\ast (s_n^{\der}) = k + (\dd k)^2
\end{equation}
for some function $k \colon V \to \BC$ satisfying $(\dd k) = (\dd g_{\boldit{a}}|_V)$.
Therefore, taking completions, we find
\begin{equation*}
    \eta^\ast \phi_E^\ast \iota_n^\ast (s_n^{\der}) = \hat{k} + (\dd \hat{k})^2 \in \HH^0(\CS^0_{\hat{U}}).
\end{equation*}
Since $(\dd h)=(\dd \hat{k})=(\dd \hat{g}_{\boldit{a}})$ it follows that $\hat{k}$ has no quadratic terms and we must have $\hat{k} - h \in (\dd \hat{g}_{\boldit{a}})^2$. The ideal $ (\dd \hat{g}_{\boldit{a}})^2$ is generated by quartics so by comparing cubic terms we deduce that
\begin{align*}
    \hat{k}_3 = \hat{g}_{\boldit{a}} \circ \dd \gamma = h_3
\end{align*}
where $\hat{k}_3, h_3$ are the cubic terms of $\hat{k}, h$ respectively. As $k$ is a polynomial, we get for the higher order terms
\begin{align*}
    \hat{k} - h = \hat{k}_{\geq 4} - h_{\geq 4} & \in (\dd \hat{g}_{\boldit{a}})^2 \\
    h_{\geq \deg k} &\in (\dd \hat{g}_{\boldit{a}})^2.
\end{align*}
Since by definition $h = \hat{g}_{\boldit{a}} \circ \gamma$, where $\hat{g}_{\boldit{a}}$ is a cubic, we may truncate the power series defining $\gamma$ at sufficiently high degree to get a polynomial isomorphism $\hat{\rho} \colon \hat{V} \to \hat{V}$ such that $h' = \hat{g}_{\boldit{a}} \circ \hat{\rho}$ still satisfies $(\dd h') = (\dd h) = (\dd \hat{g}_{\boldit{a}})$ and $h' - h$ consists of summands of $h$ of degree at least $\deg k$. By the above, $h - h' \in (\dd \hat{g}_{\boldit{a}})^2$ and therefore
\begin{equation} \label{loc 2.26}
    \eta^\ast \phi_E^\ast \iota_n^\ast (s_n^{\der}) = \hat{k} + (\dd \hat{k})^2 = h + (\dd h)^2 = h' + (\dd h')^2 \in \HH^0(\CS^0_{\hat{U}}).
\end{equation}
But$V$ is an open subscheme of an affine space and hence $\hat{\rho}$ is the completion of an isomorphism $\rho \colon V \to V$. In particular, $h'$ is the completion of the function $g_{\boldit{a}}' = g_{\boldit{a}}|_V \circ \rho \colon V \to \BC$,
\begin{equation} \label{loc 2.27}
    h' = \hat{g_{\boldit{a}}'} \colon \hat{V} \to \BC.
\end{equation}

Writing $I = (\dd g_{\boldit{a}}|_V) = (\dd g_{\boldit{a}}')$, the commutative diagram 
\[
\begin{tikzcd}[row sep=large,column sep=large]
0 \arrow{r}
& \CS_U\arrow{r}\arrow[equal]{d}
& \OO_{V}/I^2\arrow{r}{\dd}\arrow{d}{\rho^\ast} 
& \Omega_{V} / I \cdot \Omega_{V} \arrow{d}{\rho^\ast}  \\
0 \arrow{r}
& \CS_U\arrow{r}
& \OO_{V}/I^2\arrow{r}{\dd}
& \Omega_{V} / I \cdot \Omega_{V}
\end{tikzcd}
\]
and \eqref{eq crit} show that
\begin{equation} \label{loc 2.28}
\phi_E^* (s_n^{\crit}) = g_{\boldit{a}}' + (\dd g_{\boldit{a}}')^2 \in \HH^0(\CS^0_{U}).
\end{equation}

Using equations \eqref{loc 2.25}, \eqref{loc 2.26}, \eqref{loc 2.27} and \eqref{loc 2.28}, the following proposition applied to the functions $k, g_{\boldit{a}}' \colon V \to \BC$ shows that there exists a Zariski open neighbourhood $v_E \in U' \subset U$ and a smooth morphism 
\[
\phi_E' \colon U' \to U \xrightarrow{\phi_E} [U_n/\GL_n]
\]
such that 
\begin{equation*}
    (\phi_E')^\ast s_n^{\crit} = (\phi_E')^\ast \iota_n^\ast s_n^{\der} \in \HH^0(\CS^0_{U'}).
\end{equation*}

As $[E]$ varies in $\CM_n$, the morphisms $\phi_E'$ give a smooth surjective cover of $\CM_n$, on which the d-critical structures $s_n^{\crit}$ and $\iota_n^\ast s_n^{\der}$ agree. Hence we must have $s_n^{\crit} = \iota_n^\ast s_n^{\der}$, as desired.
\end{proof}

\begin{prop} \label{prop: 3.12} Let $V$ be a smooth scheme
and let $f_1, f_2 \colon V \to \BC$ be two  
functions with $(\dd f_1) = (\dd f_2) =I$ such that $U = \crit(f_1) = \crit(f_2) \subset V$. Fix a point $u \in U$.
Let $\eta \colon \hat{U} \to U$ be the formal completion of $U$ at the point $u$ and $\hat{I} = I \otimes_{\OO_V} \OO_{\hat{V}}$ the ideal of $\hat{U}$ in $\hat{V}$. 

Suppose that 
\begin{equation} \label{equality of d-crit at formal level}
f_1 + \hat{I}^2 = f_2 + \hat{I}^2 \in \HH^0(\CS^0_{\hat{U}}).
\end{equation}
Then there exists a Zariski open neighbourhood $V' \subset V$ of $u$ in $V$ such that for $U' = U \times_V V'$ we have
\[
f_1\big|_{V'} + I\big|_{V'}^2 = f_2\big|_{V'} + I\big|_{V'}^2 \in \HH^0(\CS^0_{U'}).
\]
\end{prop}

\begin{proof}
Since the problem is local in $V$, we may assume that $V$ is affine. Write $\fm$ for the maximal ideal of the closed point $0 \in \hat{V}$ and $\fm_u$ for the ideal of $u \in V$. We use $\eta \colon \hat{V} \to V$ to denote the formal completion map as well. Let $V_{\loc}$ denote the localisation of $V$ at $u$. The map $\eta$ clearly factors through the localisation morphism $V_{\loc} \to V$ by a local, faithfully flat morphism $\mu \colon \hat{V} \to V_{\loc}$.

Let $\delta = f_1 - f_2 \colon \OO_{V,u} \to \OO_{V,u}/I^2$. The assumption \eqref{equality of d-crit at formal level} implies that $\mu^\ast \delta = 0$. Therefore, since $\mu$ is faithfully flat, we must have $\delta = 0$.

Hence there exist functions $h \in I^2, g \in \OO_{V,u} \setminus \fm_u$ and a positive integer $N$ such that
\[ 
g^N(f_1 - f_2) = h \in I^2.
\]
Letting $V'$ be the non-vanishing locus of $g$ completes the proof.
\end{proof}

%%%%%%%%%%%%%%%%%%%%%%%%%%%%%%%%%%%%%%%%%%%%%%%%%%%%%%%%%%%%%%%%%%%%%%%%%%%%%%%%%%%%%%%%%%%%%%%%
\subsection{Some remarks on the proof of \Cref{thm:comparison_d-crit_1}} 
We finally collect some observations to conclude the section and suggest an alternative, stronger approach to prove \Cref{thm:comparison_d-crit_1}.

One can check that the classical truncation of the isomorphism 
$$\begin{tikzcd}
    \boldsymbol{\iota}_n \colon {\bigl[}\dSpec (Q_3)_n / \GL_n{\bigr]} \arrow{r}{\sim} & \BCal{M}_n
\end{tikzcd}$$ 
given in \eqref{description of derived M_n} is the isomorphism
$$\begin{tikzcd}
\iota_n \colon [U_n / \GL_n] \arrow{r}{\sim} & \CM_n.
\end{tikzcd}$$

The target $\BCal{M}_n$ of $\boldsymbol{\iota}_n$ admits the $-1$-shifted symplectic structure $\omega_n$. 

On the other hand, using the model for shifted symplectic forms for (affine) derived quotient stacks described in \cite[Section~1]{PTVV}, we can construct an explicit $-1$-shifted symplectic form $\omega = (\omega^0, 0, 0, \ldots)$ on ${\bigl[}\dSpec (Q_3)_n / \GL_n{\bigr]}$, combining the $-1$-shifted symplectic forms for $\BRcrit(f_n)$ and $\mathrm{B}\GL_n$, such that the induced d-critical structure on the truncation $[U_n / \GL_n]$ is $s_n^{\crit}$.

\Cref{thm:comparison_d-crit_1} will then follow from the stronger statement that $\boldsymbol{\iota}_n^\ast \omega_n = \omega$ as $-1$-shifted symplectic forms. Since the appearance of the present paper, the equality $\boldsymbol{\iota}_n^\ast \omega_n = \omega$ has been established by Katz--Shi \cite{KatzShi}.

We have chosen to pursue a local argument to obtain the equality of d-critical structures using the derived deformation theory of $\BCal{M}_n$. This gives us more control over the derived tangent complex and explicit Darboux charts, which we are comfortable with, and also avoids possible technicalities on shifted symplectic structures on derived quotient stacks. Our local approach should be applicable any time one can get a concrete handle on the formal completion of a derived stack at a point.

%%%%%%%%%%%%%%%%%%%%%%%%%%%%%%%%%%%%%%%%%%%%%%%%
%%%%%%%%%%%%%%%%%%%%%%%%%%%%%%%%%%%%%%%%%%%%%%%%
\section{The d-critical structure(s) on \texorpdfstring{$\Quot_{\BA^3}(\OO^{\oplus r},n)$}{}}

Fix integers $n\geq 0$ and $r\geq 1$. Recall from \Cref{critical_quot} the critical structure
\[
\begin{tikzcd}
\NCQuot^n_r\,\supset\, \crit(f_{r,n}) \arrow{r}{\sim} & \mathrm{Q}_{r,n} = \Quot_{\BA^3}(\OO^{\oplus r},n)
\end{tikzcd}
\]
on the Quot scheme $\mathrm{Q}_{r,n}$ of length $n$ quotients $\OO^{\oplus r} \onto E$ of the trivial rank $r$ sheaf on $\BA^3$. In this section, we compare two d-critical structures
\[
s_{r,n}^{\crit} \in \HH^0\left(\CS^0_{\crit(f_{r,n})}\right),\quad s_{r,n}^{\der} \in \HH^0\left(\CS^0_{\mathrm{Q}_{r,n}}\right),
\]
by bootstrapping the arguments in our discussion of the d-critical structure of $\CM_n$ in the preceding section.

To motivate, let $\mathbf{Q}_{r,n}$ be the derived Quot scheme \cite{DerivedQuot}. Then we have a commutative diagram
\begin{equation}
\label{diagram_main_players2}
\begin{tikzcd}[row sep = large,column sep=large]
 \mathrm{Q}_{r,n}\arrow[hook]{r}{j_{r,n}}\arrow[swap]{d}{q_{r,n}} & \mathbf{Q}_{r,n}\arrow{d}{\boldit{q}_{r,n}} \\
 \CM_n\arrow[hook]{r}{j_n} & \BCal{M}_n
\end{tikzcd}    
\end{equation}
where $j_{r,n}$ and $j_n$ are the inclusions of the underlying classical spaces and $q_{r,n}$ is the forgetful morphism taking $[\OO^{\oplus r} \onto E] \in \mathrm Q_{r,n}$ to $[E] \in \CM_n$, which can be viewed as the truncation of its derived enhancement $\boldit{q}_{r,n}$.

%%%%%%%%%%%%%%%%%%%%%%%%%%%%%%%%%%%%%%%%%%%%%%%%%%%%%%%%%%%%%%%%%%%%%%%%%%%%%%%%%%%%%%%%%%%%%%%%
\subsection{The d-critical structure $s_{r,n}^{\crit}$ coming from quiver representations} 
The d-critical structure
\[
s_{r,n}^{\crit} = f_{r,n} + (\dd f_{r,n})^2 
\]
on $\crit(f_{r,n}) \into \NCQuot^n_r$
admits a local description similar to that of $s_n^{\crit}$. We fix some notation for convenience, extending the corresponding notation used for $\CM_n$ in \Cref{subsection: s_n^crit}. 

Let $Y_{r,n} = \Rep_{(1,n)}(\widetilde{L}_3) = \End_{\BC}(\BC^n)^3 \oplus \Hom_{\BC}(\BC, \BC^n)^{r} = Y_n \oplus Z_{r,n}$, i.e.~we denote by $Z_{r,n}$ the space of $r$-tuples $(v_1,\ldots,v_r)$ of vectors $v_i \in \BC^n$. Let $f_{r,n} \colon Y_{r,n} \to \BA^1$ be the regular (and $\GL_n$-invariant) function induced by the potential $W = A[B,C]$. It is crucial to observe that this function does not interact with the component $Z_{r,n}$. According to \Cref{critical_quot}, there is a stability condition $\theta$ on $\widetilde L_3$ along with a commutative diagram
\[
\begin{tikzcd}[row sep=large,column sep=large]
& Y_{r,n}^{\theta\mathrm{-st}} \arrow[hook]{r}{\textrm{open}} \arrow[swap]{d}{/\GL_n} & Y_{r,n}\arrow{d}{f_{r,n}}\\
Y_{r,n} \sslash_\theta \GL_n\arrow[equal]{r} & \NCQuot^n_r\arrow{r}{f_{r,n}} & \BA^1
\end{tikzcd}
\]
defining the noncommutative Quot scheme, on which the function $f_{r,n}$ descends, and there is an isomorphism $\iota_{r,n}\colon \crit(f_{r,n}) \simto \mathrm Q_{r,n}$. For instance, one could take $\theta = (n,-1)$ for any fixed $n>0$.

Let ${\boldit{a}} = (a_1,\ldots, a_k)$ denote a $k$-tuple of positive integers such that $\sum_{1\leq i\leq k}a_i = n$.
Denote by $Q$ the quiver

\begin{figure}[ht]
\centering
\begin{tikzpicture}[>=stealth,->,shorten >=2pt,looseness=.5,auto]
\node (P) at (0,-1.7) {$\ldots$};
\node (Q) at (-1.5,0.05) {$\ddots$};
\node (R) at (1.5,0) {$\iddots$};
\node (V) at (-1.9,0.35) {$v_1^{(1)}$};
\node (W) at (-0.9,-0.3) {$v_r^{(1)}$};
\node (T) at (2.05,0.35) {$v_1^{(k)}$};
\node (X) at (0.8,-0.3) {$v_r^{(k)}$};
  \matrix [matrix of math nodes,
           column sep={3cm,between origins},
           row sep={3cm,between origins},
           nodes={circle, draw, minimum size=5.5mm}]
{ 
& |(A)| \infty & \\
|(B)| 1 & &  |(C)| k \\   
};
\tikzstyle{every node}=[font=\small\itshape]
\path[->] (B) edge [loop below] node {$B_1$} ()
              edge [loop right] node {$C_1$} ()
              edge [loop left] node {$A_1$} ();
\path[->] (C) edge [loop below] node {$B_k$} ()
              edge [loop right] node {$C_k$} ()
              edge [loop left] node {$A_k$} ();
\draw (A) to [bend left=13,looseness=1] (B) node [midway,above] {};
\draw (A) to [bend right=11,looseness=1] (B) node [midway,above] {};
\draw (A) to [bend left=13,looseness=1] (C) node [midway,above] {};
\draw (A) to [bend right=11,looseness=1] (C) node [midway,above] {};
\end{tikzpicture}
\end{figure}
\noindent
so that
\[
Y_{r,\boldit{a}} = \Rep_{(1,a_1,\ldots,a_k)}(Q) = \prod_{i=1}^k Y_{r,a_i} = \prod_{i=1}^k \Rep_{(1,a_i)}(\widetilde{L}_3).
\]
Consider the stability condition $\boldit{\theta} = (n,-1,\ldots,-1)$ on $Q$, so that a $Q$-representation with dimension vector $(1,a_1,\ldots,a_k)$ has slope $0$. We then have the following.
\begin{lemma}\label{lemma:stable_loci_compared}
Set $\theta_i = (a_i,-1)$ and $\boldit{\theta} = (n,-1,\ldots,-1)$. Then there are open immersions
\[
Y_{r,\boldit{a}}^{\boldit{\theta}\mathrm{-st}} \into \prod_{i=1}^k Y_{r,a_i}^{\theta_i\mathrm{-st}}    \into Y_{r,\boldit{a}}.
\]
\end{lemma}

\begin{proof}
By the very definition of $Y_{r,\boldit{a}}$, only the first inclusion needs a proof. Set $\boldit{d} = (1,a_1,\ldots,a_k)$ and note that $\boldit{d} \cdot \boldit{\theta}=0$. Any $\boldit{d}$-dimensional $Q$-representation $M$ can be written in the form $M=(M_1,\ldots,M_k)$, where $M_i \in \Rep_{(1,a_i)}(\widetilde L_3)$. We want to show that if such an $M$ is $\boldit{\theta}$-stable, then each $M_i$ is $\theta_i$-stable. So let us assume by contradiction that there exists $0\neq N_i \subsetneq M_i$ such that $(\underline{\dim}\, N_i)\cdot (a_i,-1)\geq 0$ for some $1\leq i\leq k$. Now, write $\underline{\dim}\, N_i = (\dd_\infty,\dd_i)$ where $0\leq \dd_\infty\leq 1$ and $0\leq \dd_i\leq a_i$, so the inequality we are assuming is $\dd_\infty a_i-\dd_i \geq 0$, from which it follows that $\dd_\infty=1$ (otherwise $N_i = 0$), and hence $\dd_i < a_i$ (otherwise $N_i = M_i$).

Consider now the proper nonzero subrepresentation $N = (M_1,\ldots,M_{i-1},N_i,M_{i+1},\ldots,M_k) \subset M$, whose dimension vector is $(1,a_1,\ldots,a_{i-1},\dd_i,a_{i+1},\ldots,a_k)$. We have 
\begin{align*}
    (\underline{\dim}\, N)\cdot \boldit{\theta} 
    &= (1,a_1,\ldots,a_{i-1},\dd_i,a_{i+1},\ldots,a_k)\cdot (n,-1,\ldots,-1)\\
    &= n-\left(\sum_{j\neq i} a_j\right) -\dd_i \\
    &= \sum_ja_j-\left(\sum_{j\neq i} a_j\right) -\dd_i \\
    &= a_i-\dd_i > 0,
\end{align*} contradicting stability of $M$. The result follows.
\end{proof}

It is clear from the definitions and \Cref{lemma:stable_loci_compared} that if a representation
\[
\left(A_i,B_i,C_i,v_1^{(i)},\ldots,v_r^{(i)}\right)_{1\leq i\leq k} \in Y_{r,\boldit{a}}
\]
is $\boldit{\theta}$-stable then the $n$-dimensional (unframed) representation
\[
(A_1\oplus \cdots \oplus A_k,B_1\oplus \cdots \oplus B_k,C_1\oplus \cdots \oplus C_k) \in \Rep_n(L_3)
\]
is spanned by the vectors
\[
v_j = 
\begin{pmatrix}
v_j^{(1)} \\
v_j^{(2)} \\
\vdots \\
v_j^{(k)}
\end{pmatrix} \in \BC^n,\quad 1\leq j\leq r.
\]
There is a closed embedding $\Phi_{r,\boldit{a}} \colon Y_{r,\boldit{a}} \into Y_{r,n}$, which on $Y_n$ restricts to the embedding $\Phi_{\boldit{a}} \colon Y_{\boldit{a}} \into Y_n$ (considered in \Cref{subsection: s_n^crit}) and concatenates the elements of $\prod_{i=1}^k Z_{r,a_i}$ to produce an element of $Z_{r,n}$.

The reductive algebraic group $\GL_{\boldit{a}} = \prod_{i=1}^k \GL_{a_i}$ acts on $Y_{r,\boldit{a}}$ by componentwise conjugation on the summand $Y_{\boldit{a}}$ and $\Phi_{r,\boldit{a}}$ is equivariant with respect to the inclusion $\GL_{\boldit{a}} \subset \GL_n$ by block diagonal matrices of the same kind.

On the space $Y_{r,\boldit{a}}$ we have the $\GL_{\boldit{a}}$-invariant potential
\begin{equation} \label{def of g_a_r}
g_{r,\boldit{a}} = f_{r,a_1}\oplus \cdots\oplus f_{r,a_k} \colon Y_{r,\boldit{a}} \to \BA^1, \quad (A_i,B_i,C_i,Z_i)_i \mapsto \sum_{i=1}^k \Tr A_i [B_i, C_i],
\end{equation}
where $(A_i, B_i, C_i) \in Y_{a_i}$ and $Z_i = (v_1^{(i)},\ldots,v_r^{(i)}) \in Z_{r,a_i} = \Hom_{\BC}(\BC,\BC^{a_i})^{r}$ for $i= 1,\ldots, k$. 
We have the obvious relation $f_{r,n} \circ \Phi_{r,\boldit{a}} = g_{r,\boldit{a}}$, and we observed above that the $\boldit{\theta}$-stable locus embeds in the $\theta$-stable locus, which yields a commutative diagram
\[
\begin{tikzcd}
Y_{r,\boldit{a}}^{\boldit{\theta}\mathrm{-st}} \arrow[hook]{r}\arrow[hook]{d} & Y_{r,n}^{\theta\mathrm{-st}}\arrow[hook]{d} &  \\
Y_{r,\boldit{a}}\arrow[bend right=25]{rr}[description]{g_{r,\boldit{a}}}\arrow[hook]{r}{\Phi_{r,\boldit{a}}} & Y_{r,n}\arrow{r}{f_{r,n}} & \BA^1
\end{tikzcd}
\]
where the vertical arrows are open immersions and $\theta=(n,-1)$.

If we set 
\[
U_{r,n} = \crit(f_{r,n}) \subset Y_{r,n},\quad
U_{r,\boldit{a}} = \crit(g_{r,\boldit{a}}) \subset Y_{r,\boldit{a}},
\]
the restriction $\Phi_{r,\boldit{a}} \colon U_{r,\boldit{a}} \into U_{r,n}$ is still equivariant with respect to the inclusion $\GL_{\boldit{a}} \subset \GL_n$, and so it induces a morphism of schemes
\begin{equation}
    \psi_{r,\boldit{a}} \colon U_{r,\boldit{a}} \sslash_{\boldit{\theta}} \GL_{\boldit{a}} \to U_{r,n} \sslash_{\theta} \GL_n,
\end{equation}
which fits in a diagram
\[
\begin{tikzcd}[row sep=large,column sep=large]
U_{r,\boldit{a}} \sslash_{\boldit{\theta}} \GL_{\boldit{a}} \arrow[hook]{r}{\textrm{open}}\arrow[swap]{d}{\psi_{r,\boldit{a}}}
& \prod_{i=1}^k U_{r,a_i} \sslash_{\theta_i} \GL_{a_i} \arrow{r}{\sim} & \prod_{i=1}^k \mathrm{Q}_{r,a_i}\arrow[dashed]{dl} \\
U_{r,n} \sslash_{\theta} \GL_n\arrow{r}{\sim} & \mathrm{Q}_{r,n} &
\end{tikzcd}
\]
Let $\mathrm Q_{r,\boldit{a}}^\circ \into \mathrm Q_{r,\boldit{a}} = \prod_{i=1}^k \mathrm{Q}_{r,a_i}$ be the open subscheme of $k$-tuples of quotients
\[
K_{\boldit{a}} = ([\OO^{\oplus r} \onto E_1],\ldots,[\OO^{\oplus r} \onto E_k]) \in \mathrm Q_{r,\boldit{a}}
\]
such that $\Supp(E_i) \cap \Supp(E_j) = \emptyset$ for all $i\neq j$. We can use the above diagram to identify the map $\psi_{r,\boldit{a}}$ with the `union of points' map (denoted the same way)
\[
\psi_{r,\boldit{a}} \colon \mathrm Q_{r,\boldit{a}}^\circ \to \mathrm Q_{r,n},
\]
which takes a point $K_{\boldit{a}} \in \mathrm{Q}_{r,\boldit{a}}^\circ$ as above to the joint surjection $[\OO^{\oplus r} \onto E_1 \oplus \cdots \oplus E_k] \in \mathrm Q_{r,n}$. This morphism is \'etale by \cite[Proposition A.3]{BR18}.

Let $K = [\OO^{\oplus r} \onto E] \in \mathrm{Q}_{r,n}$ be in the image of the map $\psi_{r,\boldit{a}}$, i.e.~assume we can write $K = \psi_{r,\boldit{a}}(K_{\boldit{a}})$. The local structure of the d-critical locus $(\crit(f_{r,n}), s_{r,n}^{\crit})$ around $\iota_{r,n}^{-1}(K)$ is then described as follows.

\begin{lemma} \label{etale d-critical chart s_r,n^crit}
If $\phi_{r,\boldit{a}}$ denotes the composition
\[
\begin{tikzcd}
\phi_{r,\boldit{a}} \colon U_{r,\boldit{a}}^{\boldit{\theta}\mathrm{-st}} \arrow{r} & U_{r,\boldit{a}} \sslash_{\boldit{\theta}}\GL_{\boldit{a}} \arrow{r}{\psi_{r,\boldit{a}}} & U_{r,n} \sslash_{\theta} \GL_n, 
\end{tikzcd}
\]
then we have an identity of d-critical structures
\[
\phi_{r,\boldit{a}}^* s_{r,n}^{\crit} = g_{r,\boldit{a}} + (\dd g_{r,\boldit{a}})^2 \in \HH^0(\CS_{U_{r,\boldit{a}}^{\boldit{\theta}\mathrm{-st}}}^0)^{\GL_{\boldit{a}}}.
\]
\end{lemma}

\begin{proof}
The commutative diagram
\[
\begin{tikzcd}[row sep=large,column sep=large]
U_{r,\boldit{a}} \sslash_{\boldit{\theta}} \GL_{\boldit{a}} \arrow[hook]{r} \arrow[swap]{d}{\psi_{r,\boldit{a}}}
&  Y_{r,\boldit{a}} \sslash_{\boldit{\theta}} \GL_{\boldit{a}} \arrow[swap]{d}{\Phi_{r,\boldit{a}}} \arrow{dr}{g_{r,\boldit{a}}}\\
U_{r,n} \sslash_{\theta} \GL_n \arrow[hook]{r} & Y_{r,n} \sslash_{\theta} \GL_n \arrow[swap]{r}{f_{r,n}} & \BA^1
\end{tikzcd}
\]
and the fact that $\psi_{r,\boldit{a}}$ is \'{e}tale immediately imply the claim, using the defining properties of the sheaf $\CS_{U_{r,\boldit{a}}^{\boldit{\theta}\mathrm{-st}}}^0$.
\end{proof}

%%%%%%%%%%%%%%%%%%%%%%%%%%%%%%%%%%%%%%%%%%%%%%%%%%%%%%%%%%%%%%%%%%%%%%%%%%%%%%%%%%%%%%%%%%%%%%%%
\subsection{The d-critical structure \texorpdfstring{$s_{r,n}^{\der}$}{} coming from derived symplectic geometry} \label{description of Dquot}

We start right away with the following definition.

\begin{definition}\label{def:derived_dcrit_quot}
The \emph{derived d-critical structure} on $\mathrm{Q}_{r,n}$ is defined as 
\[
s_{r,n}^{\der} = q_{r,n}^* s_n^{\der} \in \HH^0\left(\CS_{\mathrm Q_{r,n}}^0\right).
\]
\end{definition}

\begin{remark}
We observed in \Cref{main_thm_BODY} how \Cref{thm:comparison_d-crit}, combined with \Cref{prop:PB_of_sn}, proves the relation
\[
\iota_{r,n}^\ast s_{r,n}^{\der} = s_{r,n}^{\crit} \in \HH^0\left(\CS_{\crit(f_{r,n})}^0\right)
\]
that is the content of \Cref{main_thm}.
\end{remark}

In the rest of this subsection, we motivate and justify the above definition by explaining how the d-critical structure $s_{r,n}^{\der}$ arises naturally using the derived Quot scheme in Diagram~\eqref{diagram_main_players2}.

We proceed to describe the derived Quot scheme $\mathbf{Q}_{r,n}$.

Recall that as in the case of classical Quot schemes \cite{Nit}, the derived Quot scheme $\mathbf{Q}_{r,n}$ is defined as follows: Let $\BA^3 \subset \BP$ be any compactification of $\BA^3$, for example $\BP = \BP^3$. Then $\mathbf{Q}_{r,n}$ is the open derived subscheme of the derived Quot scheme $\mathbf{Q}_{\BP, r, n}$ \cite{DerivedQuot} parametrising quotients $[\OO_{\BP}^{\oplus r} \onto E]$ where $E$ is a $0$-dimensional length $n$ sheaf on $\BP$ whose support is contained in $\BA^3$.

By direct computation, we have that $\mathbf{Q}_{r,n}$ is a quasi-affine dg-scheme (or more generally derived scheme) $[\boldsymbol{\Spec}\ R / \GL_n]$ where $R$ is a sheaf of commutative dg-algebras generated in degrees $0,-1,-2$ respectively by
\begin{equation*}
    \OO_{Y_{r,n}^{\theta\mathrm{-st}}},\quad \Rep_n(L_3)^{\oplus 3} \simeq \gl_n^{\oplus 3},\quad \Rep_n(L_3) \simeq \gl_n
\end{equation*}
and the differentials are exactly the same as for the algebra $(Q_3)_n$. Notice that we are slightly abusing the notation $\boldsymbol{\Spec}\ R$ here.

It is now clear that there is an inclusion $(Q_3)_n \subset R$ of dg-algebras, where $(Q_3)_n$ is as in Equation~\eqref{description of derived M_n}. This induces precisely the forgetful map
\begin{equation*}
    \boldit{q}_{r,n} \colon \mathbf{Q}_{r,n} \lr \BCal{M}_n.
\end{equation*}
Since the map of dg-algebras is evidently smooth, it follows that $\boldit{q}_{r,n}$ is smooth.

Given these descriptions of $\mathbf{Q}_{r,n}$, $\BCal{M}_n$ and the simple ``product" structure of the map $\boldit{q}_{r,n}$, we see that while the derived Quot scheme is not $-1$-shifted symplectic, any Darboux chart of $\BCal{M}_n$ will give rise to a smooth chart of $\mathbf{Q}_{r,n}$ which is a smooth d-critical chart for the classical truncation $Q_{r,n}$. These d-critical charts induce precisely the d-critical structure $s_{r,n}^{\der}$ of Definition~\ref{def:derived_dcrit_quot}. The (closed, degenerate) $2$-form $\boldit{q}_{r,n}^\ast \omega_n$ is compatible with this d-critical structure in the appropriate sense.

%%%%%%%%%%%%%%%%%%%%%%%%%%%%%%%%%%%%%%%%%%%%%%%%
%%%%%%%%%%%%%%%%%%%%%%%%%%%%%%%%%%%%%%%%%%%%%%%%
\section{The case of a compact Calabi--Yau 3-fold}
\label{sec:compact_section}
Let $F$ be a locally free sheaf of rank $r$ on a smooth, projective Calabi--Yau $3$-fold $X$, where we have fixed a trivialisation $\zeta\colon \Lambda^3\Omega_X \simto \OO_X$. Form the Quot scheme
\[
\mathrm{Q}_{F,n}=\Quot_X(F,n).
\]
In this section we prove \Cref{thm:main_thm_B}, showing that the derived d-critical structure on $Q_{F,n}$, defined in \eqref{eqn:Pullback_to_quot} below, is locally modelled on the derived critical structure $s_{r,n}^{\der}$ of \Cref{def:derived_dcrit_quot}.

We will need the following algebraic result.

\begin{lemma}
\label{lemma:Ext_algebra}
Let $x \in X$ be a point on a smooth projective Calabi--Yau $3$-fold $X$, and let $0 \in \BA^3$ be the origin. There is an equivalence of dg-algebras 
\[
\Ext^\ast (\OO_x,\OO_x) \,\cong \,\Ext^\ast (\OO_0,\OO_0).
\]
Moreover the higher Massey products $m_n$ vanish for $n\geq 3$.
\end{lemma}

\begin{proof}
Let $X$ be \emph{any} smooth $3$-fold, $x \in X$ a point. The dg-algebra $\Ext^\ast (\OO_x, \OO_x)$ can be computed in either the algebraic or analytic category, the result and its $A_\infty$-structure being the same, capturing the deformation theory of a point inside a smooth $3$-fold. The first statement then follows after identifying suitable analytic neighbourhoods of $x \in X$ and $0 \in \BA^3$.
The vanishing of the Massey products $m_n\colon \Ext^1(\OO_x,\OO_x)^{\otimes n} \to \Ext^2(\OO_x,\OO_x)$ for $n\geq 3$ is a consequence of \Cref{lemma:dg}. 
\iffalse
Let now $x$ be a point on a Calabi--Yau $3$-fold $X$. We let 
\[
(-,-)_x\colon \Ext^j(\OO_x,\OO_x)\times \Ext^{3-j}(\OO_x,\OO_x) \to \Ext^3(\OO_x,\OO_x) \xrightarrow{\tr} \BC
\]
be the Serre duality pairing. The cyclic structure of the Ext algebra is encoded in the relations
\[
(m_2(a_1,a_2),a_3)_x = (m_2(a_2,a_3),a_1)_x
\]
for all $a_i \in \Ext^1(\OO_x,\OO_x)$, where $m_2$ agrees with the Yoneda product.
The diagram
\[
\begin{tikzcd}
\Ext^j(\OO_x,\OO_x)\times \Ext^{3-j}(\OO_x,\OO_x) \isoarrow{d}\arrow{r} & \Ext^3(\OO_x,\OO_x)\isoarrow{d}\arrow{r}{\tr_x} & \BC \isoarrow{d} \\
\Ext^j(\OO_0,\OO_0)\times \Ext^{3-j}(\OO_0,\OO_0) \arrow{r} & \Ext^3(\OO_0,\OO_0)\arrow{r}{\tr_0} & \BC
\end{tikzcd}
\]
shows that the cyclic structures agree up to a nonzero scalar. \fi
\end{proof}

\begin{remark}
    By the results of \cite{Okke}, it is actually the case that there is an equivalence $\Ext^\ast (\OO_x,\OO_x) \,\cong \,\Ext^\ast (\OO_0,\OO_0)$ as \textit{cyclic} dg-algebras.
\end{remark}

We have a commutative diagram
\[
\begin{tikzcd}
\mathrm{Q}_{F,n} \arrow[swap]{dr}{h}\arrow{r}{q} & \CM_X(n) \arrow{d}{p} \\
& \Sym^n X
\end{tikzcd}
\]
where $\CM_X(n)$ is the moduli stack of $0$-dimensional sheaves of length $n$ over $X$, the morphism $p$ is the map to the coarse moduli space (the $n$-th symmetric product $\Sym^nX = X^n/\mathfrak S_n$) and $q = q_{F,n}$ is the (smooth) forgetful morphism sending a surjection $[F \onto E]$ to the point $[E]$. The composition $h = p\circ q$ agrees with the Quot-to-Chow map \cite[Section 6]{Grothendieck_Quot}. The trivialisation $\zeta\colon \Lambda^3\Omega_X \simto \OO_X$ induces a canonical $-1$-shifted symplectic structure $\omega_{X,n}$ on $\BCal{M}_X(n)$, whose truncation $s_{X,n} = \tau(\omega_{X,n}) \in \HH^0(\CS^0_{\CM_X(n)})$ induces a d-critical structure
\begin{equation}
    \label{eqn:Pullback_to_quot}
    s_{F,n} = q^\ast s_{X,n}
\end{equation}
on the Quot scheme $\mathrm Q_{F,n}$. 

Fix a $0$-cycle $[E] \in \Sym^nX$ represented, as ever, by a polystable sheaf
\begin{equation}
\label{polystable_sheaf}
E=\bigoplus_{i=1}^k \BC^{a_i}\otimes \OO_{x_i}, \quad x_i \in X, \quad x_i \neq x_i \,\,\textrm{for}\,\,i\neq j.
\end{equation}
In this particular case, the \emph{Ext quiver} $Q_{E_\bullet}$ associated to $E$ (see Toda's paper \cite[Section 3.3]{Toda:Moduli_Ext} for a more general definition) is the quiver with vertex set $V(Q_{E_\bullet})=\set{1,2,\ldots,k}$ and edge set
\[
E(Q_{E_\bullet}) = \coprod_{1\leq i,j\leq k} E_{i,j},
\]
where $E_{i,j} \subset \Ext^1(\OO_{x_i},\OO_{x_j})^\vee$ is a $\BC$-linear basis. It follows that $E_{i,j} = \emptyset$ for $i\neq j$, and 
\begin{equation}
    \label{basis_ext1}
E_{i,i} = \set{e_{i,1},e_{i,2},e_{i,3}}\subset \Ext^1(\OO_{x_i},\OO_{x_i})^\vee
\end{equation} 
contains $3$ elements. The source and target maps $s$ and $t$ from $E(Q_{E_\bullet})$ to the vertex set $V(Q_{E_\bullet})$ both send $E_{i,i}$ to the vertex $i$. In other words, the quiver $Q_{E_\bullet}$ is a disjoint union of $k$ copies of the $3$-loop quiver (see \Cref{fig:ext_quiver}).

\begin{figure}[ht]
\centering
\begin{tikzpicture}[>=stealth,->,shorten >=2pt,looseness=.5,auto]
  \matrix [matrix of math nodes,
           column sep={3cm,between origins},
           row sep={3cm,between origins},
           nodes={circle, draw, minimum size=5.5mm}]
{ 
|(B)| 1 \\         
};
\tikzstyle{every node}=[font=\small\itshape]
\path[->] (B) edge [loop above] node {$e_{1,1}$} ()
              edge [loop right] node {$e_{1,2}$} ()
              edge [loop below] node {$e_{1,3}$} ();
\end{tikzpicture}
\qquad
\begin{tikzpicture}[>=stealth,->,shorten >=2pt,looseness=.5,auto]
  \matrix [matrix of math nodes,
           column sep={3cm,between origins},
           row sep={3cm,between origins},
           nodes={circle, draw, minimum size=5.5mm}]
{ 
|(B)| 2 \\         
};
\tikzstyle{every node}=[font=\small\itshape]
\path[->] (B) edge [loop above] node {$e_{2,1}$} ()
              edge [loop right] node {$e_{2,2}$} ()
              edge [loop below] node {$e_{2,3}$} ();
\end{tikzpicture}
\qquad 
\begin{tikzpicture}[>=stealth,->,shorten >=2pt,looseness=.5,auto]
  \matrix [matrix of math nodes,
           column sep={3cm,between origins},
           row sep={3cm,between origins},
           nodes={circle, draw, minimum size=5.5mm}]
{ 
|(B)| k \\         
};
\node at (-1.5,0) {$\cdots$};
\tikzstyle{every node}=[font=\small\itshape]
\path[->] (B) edge [loop above] node {$e_{k,1}$} ()
              edge [loop right] node {$e_{k,2}$} ()
              edge [loop below] node {$e_{k,3}$} ();
\end{tikzpicture}
\caption{The Ext quiver of a $0$-dimensional polystable sheaf \eqref{polystable_sheaf}.}\label{fig:ext_quiver}
\end{figure}
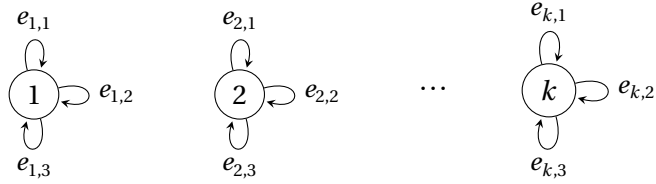

There is an associated convergent superpotential (see \cite[Sections 2.2, 2.6]{Toda:Moduli_Ext} for more details)
\[
W_{E_\bullet} \in \BC\set{Q_{E_\bullet}} \subset \BC \llbracket Q_{E_\bullet}\rrbracket,
\]
whose general definition reads as follows.
First, for each $\OO_{x_i} \in \Coh(X)$, consider the Massey products on the dg-algebra $\Ext^\ast(\OO_{x_i},\OO_{x_i})$, defined by the maps
\[
m_n \colon \Ext^1(\OO_{x_i},\OO_{x_i})^{\otimes n} \to \Ext^2(\OO_{x_i},\OO_{x_i}).
\] 
Denote by 
\[
(-,-)_{x_i}\colon \Ext^2(\OO_{x_i},\OO_{x_i}) \times \Ext^1(\OO_{x_i},\OO_{x_i}) \to \Ext^3(\OO_{x_i},\OO_{x_i}) \xrightarrow{\tr} \BC
\]
the Serre duality pairing. Then, by the Calabi--Yau condition, for any given elements $a_1,\ldots,a_n \in \Ext^1(\OO_{x_i},\OO_{x_i})$, one has the cyclicity relation
\[
(m_{n-1}(a_1,\ldots,a_{n-1}),a_n)_{x_i} = (m_{n-1}(a_2,\ldots,a_{n}),a_1)_{x_i}.
\]
Let $E_{i,i}^\vee = \set{e_{i,1}^\vee,e_{i,2}^\vee,e_{i,3}^\vee} \subset \Ext^1(\OO_{x_i},\OO_{x_i})$ be the dual basis of \eqref{basis_ext1}. Then, Toda defines in \cite[Section 5.5]{Toda:Moduli_Ext} the superpotential
\[
W_{E_\bullet} = \sum_{n\geq 3}\sum_\psi \sum_{e_i \in E_{\psi(i),\psi(i+1)}}a_{\psi,e_{\bullet}}\cdot e_1\cdots e_n,
\]
where $\psi$ runs over the set of maps $\set{1,2,\ldots,n+1} \to \set{1,\ldots,k}$ such that $\psi(1)=\psi(n+1)$, and the coefficients $a_{\psi,e_{\bullet}} \in \BC$ are defined by
\begin{equation}
    \label{psi_coefficients}
a_{\psi,e_{\bullet}}=\frac{1}{n}(m_{n-1}(e_1^\vee,\ldots,e_{n-1}^\vee),e_{n}^\vee).
\end{equation}

We now determine the (trace of the) superpotential $W_{E_\bullet}$ explicitly.

\begin{lemma}\label{lemma:trace_ext_potential}
Given a polystable sheaf $E$ as in \eqref{polystable_sheaf}, one has, up to a scalar,
\[
\Tr W_{E_\bullet} = \sum_{i=1}^k \Tr A_i[B_i,C_i],
\]
where we have set $A_i = e_{i,1}$, $B_i = e_{i,2}$ and $C_i = e_{i,3}$. In particular, $\Tr W_{E_\bullet}$ defines a regular (everywhere convergent) function on $\Rep_{\boldit{a}}(Q_{E_\bullet})$, where $\boldit{a} = (a_1,\ldots,a_k)$ is determined by \eqref{polystable_sheaf}. %Since we have an identification $\Rep_{\boldit{a}}(Q_{E_\bullet}) = Y_{\boldit{a}}$, the cubic term of $\Tr W_{E_\bullet}$ thus equals $g_{\boldit{a}}$, as defined in \Cref{def of g_a}.
\end{lemma}

\begin{proof}
By \Cref{lemma:Ext_algebra}, $m_n=0$ for $n\geq 3$, so by \eqref{psi_coefficients} only $n=3$ contributes to the sum. This already proves the last statement, about convergence. Since the only nonvanishing Ext groups $\Ext^1(\OO_{x_i},\OO_{x_j})$ are those where $i=j$, every $\psi$ in the sum satisfies $\psi(i)=\psi(i+1)$, hence the sum over $\psi$ is actually a sum over integers from $1$ up to $k$. Now, we have relations
\begin{align*}
  j_i &= (m_2(A_i^\vee,B_i^\vee),C_i^\vee)_{x_i}=(m_2(B_i^\vee,C_i^\vee),A_i^\vee)_{x_i}=(m_2(C_i^\vee,A_i^\vee),B_i^\vee)_{x_i} \\
  l_i &= (m_2(A_i^\vee ,C_i^\vee ),B_i^\vee )_{x_i} = (m_2(C_i^\vee ,B_i^\vee),A_i^\vee)_{x_i} = (m_2(B_i^\vee ,A_i^\vee ),C_i^\vee )_{x_i}.
\end{align*}
Thus
\[
W_{E_\bullet} = \sum_{i=1}^k \frac{j_i}{3}(A_iB_iC_i+B_iC_iA_i+C_iA_iB_i)+\frac{l_i}{3}(A_iC_iB_i+B_iA_iC_i+C_iB_iA_i).
\]
But since $m_2$ agrees with the Yoneda pairing, we have
\[
j_i + l_i = (m_2(A_i^\vee,B_i^\vee),C_i^\vee)_{x_i} + (m_2(B_i^\vee ,A_i^\vee),C_i^\vee)_{x_i} = (m_2(A_i^\vee,B_i^\vee)+m_2(B_i^\vee ,A_i^\vee),C_i^\vee)_{x_i} = 0.
\]
Since $j_i$ and $l_i$ do not depend on $i$, we can set $j=j_i$ and $l=l_i$, so that $l=-j$, thus
\[
W_{E_\bullet} = \frac{j}{3}\sum_{i=1}^k A_i[B_i,C_i] + B_i[C_i,A_i] + C_i[A_i,B_i].
\]
It follows that $\Tr W_{E_\bullet} = j\cdot \sum_{1\leq i\leq k} \Tr A_i[B_i,C_i]$, as required.
\end{proof}

Let $\mathfrak{M}_{\boldit{a}}(Q_{E_\bullet})$ denote the moduli stack of $\boldit{a}$-dimensional representations of the Ext quiver $Q_{E_\bullet}$.
In the diagram
\[
\begin{tikzcd}
\BA^1 & &  \Rep_{\boldit{a}}(Q_{E_\bullet}) \arrow[ll,"\Tr W_{E_\bullet}"']\arrow{r}\arrow[swap]{dr}{\pi} & \mathfrak{M}_{\boldit{a}}(Q_{E_\bullet}) \arrow{d}\arrow[equal]{r} & \prod_{1\leq i\leq k} \mathfrak{M}_{a_i}(L_3) \\
& & & M_{\boldit{a}}(Q_{E_\bullet})\arrow[equal]{r} & \prod_{1\leq i\leq k} \Sym^{a_i}\BA^3.
\end{tikzcd}
\]
we thus have a canonical identification $(\Rep_{\boldit{a}}(Q_{E_\bullet}),\Tr W_{E_\bullet}) = (Y_{\boldit{a}},g_{\boldit{a}})$, where $g_{\boldit{a}}$ was defined in \Cref{def of g_a} and it equals $\Tr W_{E_\bullet}$ by \Cref{lemma:trace_ext_potential}.

By \cite[Theorem 5.3]{Toda:2017aa}, we can find open \emph{analytic} neighbourhoods
\[
0 \in V \subset M_{\boldit{a}}(Q_{E_\bullet}),\qquad [E] \in T \subset \Sym^nX
\]
and an analytic isomorphism $i$ fitting in a diagram
\[
\begin{tikzcd}[column sep=large]
\CZ_{\boldit{a}}={\bigl[\crit(g_{\boldit{a}}|_{V})\,\big/\GL_{\boldit{a}}\bigr]} \arrow[r,"i","\sim"'] & p^{-1}(T) \arrow[hook]{r}{\textrm{open}} & \CM_X(n),
\end{tikzcd}
\]

\begin{lemma}
    Up to potential shrinking of $\pi^{-1}(V)$ around $\pi^{-1}(0)$, we have
    \begin{equation}\label{eqn:s_x,n restricted}
s_{X,n}\big|_{\CZ_{\boldit{a}}}= g_{\boldit{a}}|_{V} + (\dd g_{\boldit{a}}|_{V})^2 \,\in\,\HH^0(\CS^0_{\CZ_{\boldit{a}}}),
\end{equation}
where we are abusing notation and writing $g_{\boldit{a}}|_V$ for the restriction of $g_{\boldit{a}}$ to $\pi^{-1}(V) \subset \Rep_{\boldit{a}}(Q_{E_\bullet})$.
\end{lemma}

\begin{proof}
    The arguments in the proof of Theorem~\ref{thm:comparison_d-crit_1} and Proposition~\ref{prop: 3.12} go through verbatim, as long as we can verify that the statement of Lemma~\ref{formal etale d-critical chart s_n^der} applies to $\CM_X(n)$. But this is indeed the case, since Lemma~\ref{formal etale d-critical chart s_n^der} and its proof only depend on the minimal model of the dg-algebra $\fg_E$, where $E$ was defined in~\eqref{polystable_sheaf}. But, by Lemma~\ref{lemma:Ext_algebra}, we may assume that the points $x_i \in X$ are distinct points of $\BA^3$, so we get a formal d-critical chart of the required form at $[E]$.
\end{proof}

Now form the open subscheme
\[
\mathrm{Q}_{F,\boldit{a}} = \CZ_{\boldit{a}}\times_{\CM_{X}(n)}\mathrm{Q}_{F,n} \into \mathrm{Q}_{F,n},
\]
and consider the cartesian diagram
\begin{equation}\label{big_diagram}
\begin{tikzcd}[row sep=large,column sep=large]
& \mathrm{Q}_{F,\boldit{a}}\MySymb{dr}\arrow{d}{q}\arrow[hook]{r} & \mathrm{Q}_{F,n}\arrow{d}{q} \\
\CZ'_{\boldit{a}}\MySymb{dr}\arrow[swap]{d}{\textrm{\'{e}tale}}\arrow{r} & \CZ_{\boldit{a}}\arrow{d}{\psi_{\boldit{a}}}\arrow[hook]{r} & \CM_{X}(n) \\
\mathrm{Q}_{r,n}\arrow[swap]{r}{q_{r,n}} & \CM_n & 
\end{tikzcd}
\end{equation}
defining the scheme $\CZ'_{\boldit{a}}$. Note that, possibly after shrinking $V$, thanks to \Cref{Lemma:etaleness_psi} and \Cref{etale d-critical chart s_n^crit} we may assume $\psi_{\boldit{a}}\colon \CZ_{\boldit{a}} \to \CM_n$ is \'etale and satisfies
\begin{equation}
\label{eqn:g_a}
    \psi_{\boldit{a}}^\ast s_n^{\der} = g_{\boldit{a}}|_{V} + (\dd g_{\boldit{a}}|_{V})^2.
\end{equation}
Next, we show that the morphisms
\begin{equation}\label{smooth_maps}
\mathrm{Q}_{F,\boldit{a}} \to \CZ_{\boldit{a}}, \quad \CZ_{\boldit{a}}' \to \CZ_{\boldit{a}}
\end{equation}
look the same \'etale locally. We work locally analytically around a point $[E] \in \CZ_{\boldit{a}} \subset \CM_X(n)$. We let $[E'] \in \CM_n$ be the image of $[E]$ under the \'etale map $\psi_{\boldit{a}}$. Let 
\[
B \subset \BC^{rn}
\]
be the analytic open subset corresponding to surjective maps in
\[
\Hom_X(F,E) = \Hom_{\BA^3}(\OO^{\oplus r},E')=\BC^{rn}.
\]
We identify $B$ with the fibre of $q$ (resp.~$q_{r,n}$) over the point $[E]$ (resp.~$[E']$).
Since the morphisms \eqref{smooth_maps} are smooth, they are both analytically locally trivial (on the source), so any point in $q^{-1}([E]) \subset \mathrm{Q}_{F,\boldit{a}}$ admits an analytic open neighbourhood of the form $B \times W_{\boldit{a}}$, where $W_{\boldit{a}} \subset \CZ_{\boldit{a}}$ is a suitable analytic open neighbourhood of $[E]$. Repeating the same reasoning with the map $\CZ'_{\boldit{a}} \to \CZ_{\boldit{a}}$ and shrinking $W_{\boldit{a}}$ further if necesssary, we see that in Diagram \eqref{big_diagram} we can make the replacement
\[
\begin{tikzcd}
& \mathrm{Q}_{F,\boldit{a}} \arrow{d}{q} \\
\CZ'_{\boldit{a}} \arrow{r}{} & \CZ_{\boldit{a}} 
\end{tikzcd}
\quad \rightsquigarrow \quad
\begin{tikzcd}
& B \times W_{\boldit{a}} \arrow{d}{\pr_2} \\
B \times W_{\boldit{a}} \arrow{r}{\pr_2} & W_{\boldit{a}} 
\end{tikzcd}
\]
where $B \times W_{\boldit{a}}$ has an \'etale map (the same $\psi_{\boldit{a}}$ as above) down to $\mathrm{Q}_{r,n}$. Now we compute 
\begin{align*}
    s_{F,n} \big|_{B \times W_{\boldit{a}}} 
&=\pr_2^\ast \bigl(s_{X,n} \big|_{W_{\boldit{a}}}\bigr) & \\
&=\pr_2^\ast \bigl(g_{\boldit{a}}|_{V} + (\dd g_{\boldit{a}}|_{V})^2\bigr) & \textrm{by }\eqref{eqn:s_x,n restricted} \\
&=\pr_2^\ast \psi_{\boldit{a}}^\ast s_n^{\der} & \textrm{by } \eqref{eqn:g_a} \\
&=s_n^{\der} \big|_{B \times W_{\boldit{a}}}.
\end{align*}
This completes the proof of \Cref{thm:main_thm_B}.

%%%%%%%%%%%%%%%%%%%%%%%%%%%%%%%%%%%%%%%%%%%%%%%%
%%%%%%%%%%%%%%%%%%%%%%%%%%%%%%%%%%%%%%%%%%%%%%%%
\section{The special case of the Hilbert scheme}\label{sec:derived_hilb}

This section is devoted to the proof of \Cref{thm:symmetric_pot}.

Let $\mathrm H = \Hilb^n\BA^3$ be the Hilbert scheme parametrising $0$-dimensional subschemes of $\BA^3$ of length $n$. Denote by
\[
0 \to \mathfrak I_{\CZ} \to \OO_{\BA^3 \times \mathrm{H}} \to \OO_{\CZ} \to 0
\]
the universal short exact sequence living over $\BA^3 \times \mathrm{H}$, where $\CZ \subset \BA^3 \times \mathrm{H}$ denotes the universal subscheme. Let $\pi\colon \BA^3 \times \mathrm{H} \to \mathrm{H}$ be the projection; we use the notation $\RRlHom_\pi(-,-) = \RR \pi_\ast \RRlHom(-,-)$ throughout.

Consider the derived critical locus $\mathbf{Q} = \BRcrit(f)$ where $f = f_{1,n} \colon \NCQuot_1^n \to \BA^1$ is the potential in \Cref{critical_quot}. More precisely, $\mathbf{Q} = \dSpec B$ where $B$ is the sheaf of dg-algebras which is generated by $B^0 = \OO_{\NCQuot_1^n}$ in degree $0$ and $B^{-1} = T_{\NCQuot_1^n}$ in degree $-1$ with differential given by the dual of the section $\dd f \in \HH^0(\Omega_{\NCQuot_1^n})$. Recall that $\NCQuot_1^n = Y_{1,n}^{\theta-\mathrm{st}} / \GL_n$.

Then we have the following commutative diagram
\begin{equation}\label{diag:derived_hilb}
    \begin{tikzcd}[row sep = large, column sep = large]
\mathrm H \arrow{d}{q} \arrow[r,"\epsilon","\sim"']\arrow[bend left]{rr}[description]{\eta} &
\crit(f) \arrow[hook]{r}{j} 
& \mathbf{Q}\arrow{d}{\boldit{q}} \\
\CM_n\arrow[hook]{rr}{j_n} 
& 
& \BCal{M}_n
\end{tikzcd}
\end{equation}
where we have set $\epsilon = \iota_{1,n}^{-1}$, and $q = q_{1,n}$ is the forgetful map, whereas $j$ and $j_n$ are the inclusions of the classical spaces into their derived enhancements.

The morphism $\boldit{q}$ is obtained as follows: Recall that by \eqref{description of derived M_n} $\BCal{M}_n$ is the quotient stack $[\dSpec (Q_3)_n / \GL_n]$. For brevity, let us write $A = (Q_3)_n$ and $A^0 = (Q_3)_n^0 = \OO_{Y_n},\ A^{-1} = (Q_3)_n^{-1} = T_{Y_n}$ be the degree $0$ and degree $-1$ summands of the dg-algebra $(Q_3)_n$ respectively. There are natural morphisms $A^0 \to B^0$ and $A^{-1} \to B^{-1}$ which are $\GL_n$-invariant and, together with the trivial map $A^{-2} \to 0$, induce a morphism $\dSpec B \to \dSpec A \to [ \dSpec A / \GL_n ]$, which is the morphism $\boldit{q}$.

Note that, by definition,
\[
\begin{tikzcd}
\BE_{\crit} = j^\ast \BL_{\mathbf{Q}} \arrow{r}{\varphi_{\crit}} & \BL_{\crit(f)}
\end{tikzcd}
\]
is the critical obstruction theory on $\crit(f)$, denoted $\BE_{f}$ in the introduction, see \eqref{eqn:crit_pot_1231}.
The maps $\eta$ and $\boldit{q}$ induce a commutative diagram
\[
\begin{tikzcd}
& & \eta^\ast \boldit{q}^\ast \BL_{\BCal{M}_n} \arrow[swap]{dl}{\psi}\arrow{dd}  \\
\epsilon^\ast \BE_{\crit}\arrow[equal]{r} & \eta^\ast \BL_{\mathbf{Q}}\arrow[swap]{dr}{\epsilon^\ast \varphi_{\crit}} & \\
& & \BL_{\mathrm{H}} 
\end{tikzcd}
\]
which after applying the truncation functor $\tau_{[-1,0]}$ becomes
\begin{equation}\label{triangle}
\begin{tikzcd}
& \eta^\ast \tau_{[-1,0]} \boldit{q}^\ast \BL_{\BCal{M}_n} \arrow[swap]{dl}{\overline{\psi}}\arrow{dd}{\varphi} \\
\epsilon^\ast \BE_{\crit}\arrow[swap]{dr}{\epsilon^\ast \varphi_{\crit}} & \\
& \BL_{\mathrm{H}}
\end{tikzcd}
\end{equation}
where $\BE_{\crit}$ is unchanged since it already has cohomological amplitude in degrees $[-1,0]$. The morphism $\overline{\psi}$ is an isomorphism because it is the pullback along $\eta$ of the second of the isomorphisms \eqref{iso_tangent} established in the following lemma.

\begin{lemma}
The derivative of $\boldit{q}$ induces isomorphisms
\begin{equation}\label{iso_tangent}
\begin{tikzcd}
\BT_{\mathbf{Q}} \arrow{r}{\sim} & \tau_{[0,1]} \boldit{q}^\ast \BT_{\BCal{M}_n}\\
\tau_{[-1,0]} \boldit{q}^\ast \BL_{\BCal{M}_n} \arrow{r}{\sim} & \BL_{\mathbf{Q}}.
\end{tikzcd}
\end{equation}
Moreover, one has an isomorphism
\begin{equation}\label{truncated_tangent}
\begin{tikzcd}
\gamma\colon \tau_{[0,1]}\RRlHom_\pi(\OO_{\CZ}[-1],\OO_{\CZ}) \arrow{r}{\sim} &  \tau_{[0,1]} \eta^\ast \boldit{q}^\ast \BT_{\BCal{M}_n}.
\end{tikzcd}
\end{equation}
\end{lemma}

\begin{proof}
The derived tangent complex of $\mathbf{Q}$ is given by
\begin{equation*}
\begin{tikzcd}[column sep=large]
    \BT_{\mathbf{Q}} = [ T_{\NCQuot_1^n} \arrow{r}{\Hess(f)} & \Omega_{\NCQuot_1^n} ]
\end{tikzcd}
\end{equation*}
where $\Hess(f)$ denotes the Hessian of $f$.

Since $\NCQuot_1^n = Y_{1,n}^{\theta-\mathrm{st}} / \GL_n$ and the action of $\GL_n$ on $Y_{1,n}^{\theta-\mathrm{st}}$ is free, we may write this as the quasi-isomorphic $\GL_n$-equivariant three-term complex on $Y_{1,n}^{\theta-\mathrm{st}}$ in degrees $-1$ to $1$ (where we are slightly abusing notation by omitting certain pullbacks)
\begin{equation*}
\begin{tikzcd}[column sep=large]
    \BT_{\mathbf{Q}} = [ \gl_n \arrow{r} & T_{Y_{1,n}^{\theta-\mathrm{st}}} \arrow{r}{\Hess(f)} & \Omega_{\NCQuot_1^n} ].
\end{tikzcd}
\end{equation*}
By the definition of $\boldit{q}$, we have (again slightly abusing notation)
\begin{equation*}
\begin{tikzcd}
    \boldit{q}^\ast \BT_{\BCal{M}_n} = [ \gl_n \arrow{r} & T_{Y_{n}} \arrow{r} & \Omega_{Y_{n}} \arrow{r} & \gl_n ]
\end{tikzcd}
\end{equation*}
and the derivative $\dd \boldit{q}$ is the morphism of complexes
\begin{equation*}
\begin{tikzcd}[column sep=large]
     \gl_n \arrow{r} \ar[equal]{d} & T_{Y_{1,n}^{\theta-\mathrm{st}}} \arrow{r}{\Hess(f)} \arrow{d} & \Omega_{\NCQuot_1^n} \arrow{d} \\
     \gl_n \arrow{r} & T_{Y_{n}} \arrow{r} & \Omega_{Y_{n}} \arrow{r} & \gl_n
\end{tikzcd}
\end{equation*}
We may identify $T_{Y_{1,n}^{\theta-\mathrm{st}}}$ with $(\gl_n^{\oplus 3} \oplus \BC^n) \otimes \OO_B$ and $T_{Y_n}$ with $\gl_n^{\oplus 3} \otimes \OO_B$ so that the middle vertical arrow is the natural projection map.

It is clear that $\HH^{0}(\dd \boldit{q})$ is surjective.

To show that it is injective, notice that the leftmost arrow $\gl_n \to T_{Y_{1,n}^{\theta-\mathrm{st}}}$ maps $X \in \gl_n$ to $$([X,A], [X,B], [X,C], Xv) \in \gl_n^{\oplus 3} \oplus \BC^n$$ at the point $(A,B,C,v) \in Y_{1,n}^{\theta-\mathrm{st}}$. Since, by stability, $v$ is a cyclic vector with respect to the action of the matrices $A,B,C$, for any $w \in \BC^n$ there exists a polynomial $f(A,B,C)$ such that $w = f(A,B,C)v$. Letting $X = f(A,B,C)$, the image of $X$ is then $(0,0,0,w) \in \gl_n^{\oplus 3} \oplus \BC^n$. Therefore, the composition $\gl_n \to T_{Y_{1,n}^{\theta-\mathrm{st}}} = \gl_n^{\oplus 3} \oplus \BC^n \to \BC^n$ is fibrewise surjective and hence is a surjective morphism of locally free sheaves and the injectivity of $\HH^0(\dd \boldit{q})$ follows readily.

A similar argument involving duals shows that $\HH^1(\dd \boldit{q})$ is an isomorphism as well.

Thus the truncation $\tau_{[0,1]} \dd \boldit{q}$ induces the first isomorphism in \eqref{iso_tangent}. The second isomorphism is obtained analogously.

\smallbreak
For \eqref{truncated_tangent}, since the universal complex of $\BCal{M}_n$ over the image of $\boldit{q}$ restricts to $\OO_{\CZ}$ on $\CM_n$, we have that $\RRlHom_\pi(\OO_{\CZ}[-1], \OO_{\CZ})$ is isomorphic to $q^* j_n^* \BT_{\BCal{M}_n} = \eta^* \boldit{q}^\ast \BT_{\BCal{M}_n}$, so applying the truncation $\tau_{[0,1]}$ gives the desired isomorphism.
\end{proof}

We note that by the above explicit description of $\BT_{\BCal{M}_n}$ and the morphism $\boldit{q}$ we have an isomorphism between the complexes $\eta^* \tau_{[0,1]} \boldit{q}^\ast \BT_{\BCal{M}_n}$ and $\tau_{[0,1]} \eta^* \boldit{q}^\ast \BT_{\BCal{M}_n}$, which we use to identify the two complexes from now on.

Our final goal is to produce an isomorphism
\[
\begin{tikzcd}
\rho\colon \eta^\ast\tau_{[-1,0]} \boldit{q}^\ast \BL_{\BCal{M}_n} 
\arrow{r}{\sim} & 
\BE_{\der} = \RRlHom_\pi(\mathfrak I_{\CZ},\mathfrak I_{\CZ})_0[2]
\end{tikzcd}
\]
such that 
\begin{equation}\label{comm_POT}
\begin{tikzcd}[column sep=large,row sep=large]
\eta^\ast\tau_{[-1,0]} \boldit{q}^\ast \BL_{\BCal{M}_n} 
\arrow{r}{\rho}\arrow[swap]{d}{\varphi} & 
\BE_{\der} \arrow{dl}{\varphi_{\der}} \\
\BL_{\mathrm{H}} & 
\end{tikzcd}
\end{equation}
commutes, where $\varphi_{\der}$ is obtained from the Atiyah class of $\mathfrak I_{\CZ}$ via a classical construction (see \cite{HT}, \cite{Quot19} or \cite{Equivariant_Atiyah_Class} for full details).

First of all, we have a diagram 
\[
\begin{tikzcd}[column sep=large,row sep=large]
  0 \arrow[dashed]{r}\arrow[dashed]{d} 
& \RR\pi_\ast \OO_{\BA^3 \times \mathrm{H}}[1]\arrow[equal]{r}\arrow{d}
& \RR\pi_\ast \OO_{\BA^3 \times \mathrm{H}}[1]\arrow{d} \\
  \RRlHom_\pi(\mathfrak I_{\CZ},\OO_{\CZ}) \arrow{r}\arrow[dashed,swap]{d}{\id} 
& \RRlHom_\pi(\mathfrak I_{\CZ},\mathfrak I_{\CZ})[1]\arrow{r}\arrow{d}
& \RRlHom_\pi(\mathfrak I_{\CZ},\OO_{\BA^3\times \mathrm{H}})[1]\arrow{d} \\
  \RRlHom_\pi(\mathfrak I_{\CZ},\OO_{\CZ}) \arrow{r}
& \RRlHom_\pi(\mathfrak I_{\CZ},\mathfrak I_{\CZ})_0[1]\arrow{r}
& \RRlHom_\pi(\OO_{\CZ},\mathfrak \OO_{\BA^3\times \mathrm{H}})[2]
\end{tikzcd}
\]
in the derived category of $\mathrm{H}$, where rows and columns are exact triangles; more precisely:
\begin{enumerate}
    \item the middle column is obtained by applying $\RR\pi_\ast$ to the shifted \emph{dual} of the exact triangle $\RRlHom(\mathfrak I_{\CZ},\mathfrak I_{\CZ})_0 \to \RRlHom(\mathfrak I_{\CZ},\mathfrak I_{\CZ}) \to \OO_{\BA^3\times \mathrm{H}}$, exploiting the fact that all three objects are self-dual;
    \item the right column is obtained by applying $\RRlHom_\pi(-,\OO_{\BA^3\times \mathrm{H}})$ to the exact triangle $\OO_{\CZ}[-2] \to \mathfrak I_{\CZ}[-1] \to \OO_{\BA^3\times \mathrm{H}}[-1]$,
    \item the middle row is obtained by applying $\RRlHom_\pi(\mathfrak I_{\CZ},-)$ to the exact triangle $\OO_{\CZ} \to \mathfrak I_{\CZ}[1] \to \OO_{\BA^3\times \mathrm{H}}[1]$.
\end{enumerate}
The last row
induces an isomorphism
\[
\begin{tikzcd}
\alpha\colon \tau_{[0,1]} \RRlHom_\pi(\mathfrak I_{\CZ},\OO_{\CZ}) \arrow{r}{\sim} & \RRlHom_\pi(\mathfrak I_{\CZ},\mathfrak I_{\CZ})_0[1].
\end{tikzcd}
\]
On the other hand, the exact triangle
\[
\begin{tikzcd}
\OO_{\BA^3 \times \mathrm{H}}[-1] \arrow{r} & \OO_{\CZ}[-1] \arrow{r} & \mathfrak I_{\CZ}
\end{tikzcd}
\]
induces, via $\RRlHom_\pi(-,\OO_{\CZ})$,  an exact triangle
\[
\begin{tikzcd}
\RRlHom_\pi(\mathfrak I_{\CZ},\OO_{\CZ}) \arrow{r} & \RRlHom_\pi(\OO_{\CZ},\OO_{\CZ})[1] \arrow{r} & \RRlHom_\pi(\OO_{\BA^3\times \mathrm{H}},\OO_{\CZ})[1]
\end{tikzcd}
\]
such that the truncation of the first arrow
\[
\begin{tikzcd}
\beta\colon \tau_{[0,1]} \RRlHom_\pi(\mathfrak I_{\CZ},\OO_{\CZ}) \arrow{r}{\sim} &  \tau_{[0,1]} \RRlHom_\pi(\OO_{\CZ},\OO_{\CZ})[1]
\end{tikzcd}
\]
is an isomorphism.
Summing up, as proved in \cite[Proposition 2.2]{Cao-Kool2}, we have an isomorphism 
\[
\begin{tikzcd}
\beta\circ\alpha^{-1}\colon \RRlHom_\pi(\mathfrak I_{\CZ},\mathfrak I_{\CZ})_0[1] \arrow{r}{\sim} & \tau_{[0,1]} \RRlHom_\pi(\OO_{\CZ}[-1],\OO_{\CZ})
\end{tikzcd}
\]
along with the identifications
\begin{align*}
\eta^\ast\tau_{[-1,0]} \boldit{q}^\ast \BL_{\BCal{M}_n} 
    \, &= \, \left( \tau_{[0,1]} \eta^\ast \boldit{q}^\ast \BT_{\BCal{M}_n}\right)^\vee \\
    \, &= \, \left(\tau_{[0,1]}\RRlHom_\pi(\OO_{\CZ}[-1],\OO_{\CZ})\right)^\vee & \textrm{via }\gamma^\vee\textrm{ from }\eqref{truncated_tangent} \\
    \, &= \, \left(\RRlHom_\pi(\mathfrak I_{\CZ},\mathfrak I_{\CZ})_0[1]\right)^\vee & 
    \textrm{via }(\beta\circ\alpha^{-1})^\vee \\
    \, &= \, \BE_{\der}  & 
    \textrm{by Serre duality.}
 \end{align*}   
 We have thus obtained 
 \[
 \begin{tikzcd}
 \rho = (\beta\circ\alpha^{-1})^\vee \circ \gamma^\vee \colon \eta^\ast\tau_{[-1,0]} \boldit{q}^\ast \BL_{\BCal{M}_n} \arrow{r}{\sim} & \BE_{\der}
 \end{tikzcd}
 \]
Thus all that remains to check is the commutativity of Diagram \eqref{comm_POT}; that is, we need to check that ``up to $\rho$'' the vertical map 
\[
\begin{tikzcd}
\varphi\colon \eta^\ast\tau_{[-1,0]} \boldit{q}^\ast \BL_{\BCal{M}_n}\arrow{r} & \BL_{\mathrm{H}}
\end{tikzcd}
\]
in Diagram \eqref{triangle} agrees with the symmetric obstruction theory $\varphi_{\der}\colon \BE_{\der} \to \BL_{\mathrm{H}}$ on the Hilbert scheme, viewed as a moduli space of ideal sheaves.

\smallbreak
Recall that the \emph{Atiyah class} of a perfect complex $P$ on a scheme $Y$ is an element
\[
\At_P \,\in\, \Ext^1(P,P\otimes \BL_Y[1]) = \Hom(P[-1],P\otimes \BL_Y).
\]
Working over $Y = \BA^3 \times \mathrm{H}$, by projecting along $\pi$ we will view our Atiyah classes as elements in the group
\[
\At_P \in \Hom(P[-1],P\otimes \pi^\ast \BL_{\mathrm{H}}).
\]
By the functoriality of Atiyah classes, we have a commutative diagram
\[
\begin{tikzcd}[row sep=large,column sep=large]
  \OO_{\BA^3\times \mathrm{H}}[-1] \arrow{r}\arrow{d}{0=\At_{\OO}}
& \OO_{\CZ}[-1]\arrow{r}\arrow{d}{\At_{\OO_{\CZ}}}
& \mathfrak I_{\CZ}\arrow[dotted]{dl}[description]{u} \arrow{d}{\At_{\mathfrak I_{\CZ}}} \\
  \OO_{\BA^3\times \mathrm{H}} \otimes \pi^\ast \BL_{\mathrm{H}}\arrow{r}
& \OO_{\CZ}\otimes \pi^\ast \BL_{\mathrm{H}}\arrow{r}
& \mathfrak I_{\CZ}[1] \otimes \pi^\ast \BL_{\mathrm{H}}
\end{tikzcd}
\]
where we observed that the Atiyah class of the trivial line bundle vanishes, to give the diagonal arrow 
\[
u \colon \mathfrak I_{\CZ} \to \OO_{\CZ} \otimes \pi^\ast \BL_{\mathrm{H}}.
\]
The commutativity of the two triangles forming the right square translates into a commutative diagram
\[
\begin{tikzcd}[row sep=large,column sep=large]
\RRlHom(\mathfrak I_{\CZ},\mathfrak I_{\CZ}[1])^\vee \arrow[swap]{d}{\OO_{\CZ} \to \mathfrak I_{\CZ}[1]}\arrow[bend left=15]{dr}{\At_{\mathfrak I_{\CZ}}} & \\
\RRlHom(\mathfrak I_{\CZ},\OO_{\CZ})^\vee \arrow{r}{u} & \pi^\ast \BL_{\mathrm{H}} \\
\RRlHom(\OO_{\CZ}[-1],\OO_{\CZ})^\vee \arrow[swap,,bend right=15]{ur}{\At_{\OO_{\CZ}}}\arrow{u}{\OO_{\CZ}[-1] \to \mathfrak I_{\CZ}}
\end{tikzcd}
\]
After dualising and applying $\tau_{[0,1]}\circ \RR\pi_\ast$, we obtain a commutative diagram
\[
\begin{tikzcd}[row sep=large]
\RRlHom_\pi(\mathfrak I_{\CZ},\mathfrak I_{\CZ})_0[1] & & \\
\tau_{[0,1]}\RRlHom_\pi(\mathfrak I_{\CZ},\OO_{\CZ})\arrow[u,"\alpha"]\arrow[swap]{d}{\beta} & \RR\pi_\ast \pi^{\ast}\BT_{\mathrm{H}}\arrow{l} & \BT_{\mathrm{H}}\arrow{l}\arrow[bend right=15,swap]{ull}[description]{\varphi_{\der}^\vee}\arrow{dd}{\varphi^\vee}\arrow[bend left=15]{dll}[description]{\pi_\ast\At_{\OO_{\CZ}}^\vee} \\
\tau_{[0,1]}\RRlHom_\pi(\OO_{\CZ}[-1],\OO_{\CZ})\arrow[swap]{rrd}{\gamma} & & \\
& & \tau_{[0,1]} \eta^\ast \boldit{q}^\ast \BT_{\BCal{M}_n} 
\end{tikzcd}
\]
where $\gamma$ is the isomorphism \eqref{truncated_tangent}, and the lower right part of the diagram, stating that the composition $\gamma \circ \pi_\ast\At_{\OO_{\CZ}}^\vee$ agrees with the map $\varphi^\vee$ dual to the map $\varphi$ appearing in Diagram \eqref{triangle}, is an immediate consequence of \cite[Appendix~A]{DerivedDet} or \cite[Proposition~2.4.7]{TangentLieAlgebra}.

Dualising back, we conclude the proof of \Cref{thm:symmetric_pot}. Therefore we have proved Conjecture 9.9 in \cite{FMR}.

%%%%%%%%%%%%%%%%%%%%%%%%%%%%%%%%%%%%%%%%%%%%%%%%

\bibliographystyle{amsplain-nodash}
\bibliography{bib}

\end{document}